\documentclass[12pt]{amsart}
\usepackage{amsmath,eucal,verbatim,amsopn,amsthm,amsfonts,amssymb,mathdots,mathrsfs,esint,mathtools,enumitem,bbm}
\usepackage{amsmath,amssymb,latexsym}
\usepackage{url}
\usepackage{color}
\usepackage{amscd}
\usepackage[margin=1in]{geometry}

\DeclareMathOperator{\Real}{Re}

\newcommand{\GL}{{\mathrm{GL}}}

\newcommand{\Ind}{{\mathrm{Ind}}}

\newcommand{\I}{{\mathrm{I}}}

\newcommand{\SO}{{\mathrm{SO}}}

\newcommand{\Sp}{{\mathrm{Sp}}}

\newcommand{\wt}{\widetilde}

\def\diag{{\rm diag}}

\newcommand{\whittaker}[2]{\ensuremath{\mathcal{W}(#1,#2)}}

\newcommand{\absdet}[1]{\ensuremath{|\det{#1}|}}
\newcommand{\half}{\ensuremath{\frac12}}

\newcommand{\lmodulo}[2]{\ensuremath{#1 \setminus #2}}
\newcommand{\rconj}[1]{\ensuremath{{}^{#1}}}

\newcommand{\transpose}[1]{\ensuremath{\rconj{t}#1}}

\newcommand{\C}{\ensuremath{\mathbb{C}}}

\newcommand{\glnrep}{\ensuremath{\tau}}
\newcommand{\heisenembedding}{\ensuremath{}}
\newcommand{\intertwining}[2]{\ensuremath{M(#1,#2)}}
\newcommand{\nintertwining}[2]{\ensuremath{M^*(#1,#2)}}

\newtheorem{thm}{Theorem}[section]

\newtheorem{lem}[thm]{Lemma}
\newtheorem{prop}[thm]{Proposition}

\newtheorem {ques/conj}[thm]{Question/Conjecture}

\newtheorem{rmk}[thm]{Remark}

\begin{document}

\title[The Langlands parameter: symplectic groups]{The Langlands parameter of a simple supercuspidal representation: symplectic groups}
\author{Moshe Adrian}
\author{Eyal Kaplan}
\address{Adrian:  Department of Mathematics, Queens College, CUNY, Queens, NY 11367-1597}
\email{moshe.adrian@qc.cuny.edu}
\address{Kaplan: Department of Mathematics, Bar Ilan University, Ramat Gan 5290002, Israel}
\email{kaplaney@gmail.com}


\begin{abstract}
Let $\pi$ be a simple supercuspidal representation of the symplectic group $\Sp_{2l}(F)$, over a $p$-adic field $F$.
In this work, we explicitly compute the Rankin--Selberg $\gamma$-factor of rank-$1$ twists of $\pi$. We then completely determine the Langlands parameter of $\pi$, if $p \neq 2$. In the case that $F = \mathbb{Q}_2$, we give a conjectural description of the functorial lift of $\pi$, with which, using a recent work of Bushnell and Henniart, one can obtain its Langlands parameter.
\end{abstract}
\maketitle

\section{introduction}
In this work we compute the Langlands parameter, for the simple supercuspidal representations of symplectic groups
recently constructed by Gross and Reeder \cite{GR10} and Reeder and Yu \cite{RY14}. Our approach involves
the theory of lifting of representations from classical groups to general linear groups (\cite{CKPS2,CKPS,CPSS}),
and the definition of the $\gamma$-factor via the theory of Rankin--Selberg integrals
(\cite{Soudry,Soudry3,Soudry2,Kaplan2013a,Kaplan2015}). To some extent, this paper is a follow-up to \cite{Adrian2016},
where the first named author determined the Langlands parameter for simple supercuspidal representations of
odd orthogonal groups.

Let $F$ be a local $p$-adic field and $\Sp_{2l}=\Sp_{2l}(F)$ be the symplectic group of rank $l$ over $F$.
Fix a maximal unipotent subgroup $N_l$ of $\Sp_{2l}$, an additive character $\psi$ of $F$ of level $1$, and
an extension $\psi_{N_l}$ of $\psi$ to a non-degenerate character of $N_l$.
Consider first an arbitrary supercuspidal representation $\pi$ of $\Sp_{2l}$, where in this work supercuspidal representations are irreducible by definition. Assume $\pi$ is $\psi_{N_l}$-generic. By Cogdell \textit{et. al.} \cite[Proposition~7.2]{CKPS} there exists a unique irreducible generic representation $\Pi$ of $\GL_{2l+1}$, such that the $\gamma$-factor of $\pi\times\tau$
of Shahidi \cite{Sh3}, and the $\gamma$-factor of $\Pi\times\tau$ of Jacquet \textit{et. al.} \cite{JPSS} (or \cite{Shahidi1983}),
are identical for every supercuspidal representation $\tau$ of $\GL_n$.
This lift $\Pi$ corresponds to the standard $L$-homomorphism ${}^L \Sp_{2l} \rightarrow {}^L \GL_{2l+1}$.
Then the Langlands correspondence for $\GL_{2l+1}$ (\cite{Hn,HT2001})
can in theory be used to determine the Langlands parameter $\varphi_{\pi}$ of $\pi$, from the parameter of $\Pi$, but the construction in \cite{CKPS} does not provide sufficient information on $\Pi$ for this purpose.

While describing the Langlands correspondence explicitly is in general a difficult task, in the simple supercuspidal setting it turns out to be tractable. Assume now that $\pi$ is simple supercuspidal. Then it is always $\psi_{N_l}$-generic with respect to a certain extension $\psi_{N_l}$, hence the lift $\Pi$ exists ($\psi_{N_l}$ is given in Section~\ref{simplesupercuspidaldefinition}). The key to determine $\varphi_{\pi}$ is the cuspidal support $\mathrm{supp}(\Pi)$ of $\Pi$ (see \cite{BushnellKutzko1998} for the definition of cuspidal support): with an understanding of $\mathrm{supp}(\Pi)$, we can use the explicit local Langlands correspondence for simple supercuspidal representations of general linear groups to determine $\varphi_{\pi}$ (see \cite{BH10, BH14,AL14}).
Since $\Pi$ is in our case self-dual, and when $\tau$ is supercuspidal, the results of \cite{JPSS}, in particular the multiplicative properties, imply that $\gamma(s,\Pi\times\tau,\psi)$ has a pole at $s=1$ if and only if $\tau$ appears in $\mathrm{supp}(\Pi)$. In this manner we can construct $\mathrm{supp}(\Pi)$, up to multiplicities. Of course, in practice we work with $\gamma(s,\pi\times\tau,\psi)$, because $\pi$ (and not $\Pi$) is the explicit object.

The $\gamma$-factor for $\pi\times\tau$ can also be defined using the method of Rankin--Selberg integrals. This was carried out in
\cite{Kaplan2015} (for $\Sp_{2l}\times\GL_n$), where it was also proved that the definition coincides with that of Shahidi \cite{Sh3}.
The Rankin--Selberg definition turns out to be advantageous for the type of computations presented here, and will be our main tool.

According to the parametrization of simple supercuspidal representations (\cite{GR10,RY14}), we can identify $\pi$ (uniquely) with a triplet consisting of a uniformizer $\varpi$, a character $\omega$ of the center of $\Sp_{2l}$, and a representative $\alpha$ of a square class in the multiplicative group of the residue field of $F$. We then denote $\pi=\pi_{\alpha}^{\omega}$, and the uniformizer $\varpi$ will be suppressed from the notation throughout. See Section~\ref{simplesupercuspidaldefinition} for a precise definition and more details. We first prove:

\begin{thm}\label{thm:main}
Assume $p \neq 2$. Let $\pi=\pi_{\alpha}^{\omega}$, and let $\tau$ be a tamely ramified quadratic character of $F^{\times}$. Then $\gamma(s, \pi\times \tau, \psi_{\alpha})$ has a pole at $s = 1$
if and only if $\tau$ is the unique non-trivial quadratic character of $F^{\times}$ such that $\tau(\varpi) =  \gamma_{\psi_{\alpha}}(\varpi)^{-1}  \frac{|G(\psi^{-1}_{\alpha})|}{G(\psi^{-1}_{\alpha})}$ and $\tau|_{\mathfrak{o}^{\times}} \equiv \gamma_{\psi_{\alpha}}|_{\mathfrak{o}^{\times}}$.
Here $\gamma_{\psi_{\alpha}}$ is the Weil index of $\psi_{\alpha}$ (see Section~\ref{subsection:The groups and the Shimura integral}), $\mathfrak{o}^{\times}$ is the group of units of $F$ and
$G(\psi)$ is a certain Gauss sum (see Lemma~\ref{lemma:computation Psi M(tau,s)f_s with p not 2}).
\end{thm}

According to the theorem, $\mathrm{supp}(\Pi)$ contains a unique character $\tau_{\alpha}$, which may appear with
some finite multiplicity. Note that this character does not depend on $\omega$.
By a recent result of Blondel \textit{et. al.} \cite{BHS17}, $\varphi_{\pi}$ decomposes as a direct sum $\varphi_1 \oplus \varphi_2$, where $\varphi_1$ is an irreducible $2l$-dimensional representation of the Weil group $W_F$, and $\varphi_2$ is a $1$-dimensional representation of $W_F$; in particular, $\varphi_2 = \tau_{\alpha}$.   As we will show (see Section~\ref{fullparameter}), the properties of the $\gamma$-factor and the work of \cite{BHS17} imply that $\varphi_{\pi}$ is determined by $\tau_{\alpha}$ and by the computation of
$\gamma(s, \pi, \psi_{\alpha})$. We emphasize here that it is not enough to locate the pole in $\gamma(s, \pi\times \tau, \psi_{\alpha})$ and use the work of \cite{BHS17}; one needs additional information, and this can be obtained from $\gamma(s, \pi, \psi_{\alpha})$.

We turn to describe the parameter, first in the case that $p>2$ and $p \nmid l$.
We construct a totally ramified extension $E$ of $F$, and a character $\xi_{\alpha}^{\omega}$ of $E^{\times}$.
This character depends on several parameters including $\omega$, $\alpha$ and $\varpi$, and also $\psi$, $\tau_{\alpha}$
and the Langlands constant $\lambda_{E/F}(\psi)$. For more details and precise notation see Section~\ref{fullparameter}.
The parameters $\varphi_1$ and $\varphi_{\pi}$ are given by the following theorem, which is our main result:
\begin{thm}\label{thm:main param}
Assume $p$ is odd and $p \nmid l$. Then
\begin{align*}
\varphi_1 = \mathrm{Ind}_{W_E}^{W_F} \xi_{\alpha}^{\omega},\qquad \varphi_{\pi} = \mathrm{Ind}_{W_E}^{W_F} \xi_{\alpha}^{\omega} \oplus \tau_{\alpha}.
\end{align*}
\end{thm}

In the case that $2<p|l$, the Langlands parameter of $\pi$ can be explicitly written down using our work, the work of \cite{BHS17}, and the work of \cite{BH14}.  See Remark~\ref{rmk:fullparameter} for more details.

Thus far we have excluded dyadic fields. For $F = \mathbb{Q}_2$, there is a unique simple supercuspidal representation $\pi$ of $\Sp_{2l}$. We prove:
\begin{thm}\label{thm:main2}
Assume $\tau$ is a tamely ramified character of $F^{\times}$. Then $\gamma(s,\pi\times\tau,\psi) = \tau(2) 2^{1/2-s}$.
\end{thm}

In this case it is expected that the functorial lift of $\pi$ to $\GL_{2l+1}$ is supercuspidal. Therefore $\mathrm{supp}(\Pi)=\{\Pi\}$, and in particular $\gamma(s,\pi\times\tau,\psi)$ should be entire (as we show). Taking this for granted, the results of \cite{BH14, AL14} imply that this lift is already determined by Theorem~\ref{thm:main2}. More specifically, the functorial lift of $\pi$ to $\GL_{2l+1}$ is the unique simple supercuspidal representation $\Pi$ with a trivial central character, such that $\gamma(s, \Pi,\psi) = 2^{1/2-s}$. One can write explicit inducing data for $\Pi$, see \cite{BH14, AL14}, and then obtain $\varphi_{\pi}$ using \cite{BH14} (in fact, assuming the lift is supercuspidal, the weaker version of Theorem~\ref{thm:main2} with just $\tau=1$ suffices).

As mentioned above, we use the Rankin--Selberg definition of the $\gamma$-factor for $\pi\times\tau$. This definition is based on
a global and local theory of integral representations, with strong local uniqueness properties.
The integrals for $\Sp_{2l}\times\GL_n$ are called Shimura type integrals. The integrals for a pair of irreducible automorphic cuspidal globally generic representations (globally), or generic representations (locally), of $\Sp_{2l}$ and $\GL_n$, were introduced by Ginzburg \textit{et. al.} \cite{GRS4}. They developed the global theory and used the integrals to study the partial $L$-function. In particular, they computed the local integrals with unramified data. The $\gamma$-factor $\gamma(s,\pi\times\tau,\psi)$, for a pair of generic representations $\pi$ and $\tau$, over a local field, was defined in \cite{Kaplan2015} as the proportionality factor between two integrals $\Psi(W,\phi,f_s)$ and $\Psi^*(W,\phi,f_s)$. The data for the integrals consist of a Whittaker function $W$ in the Whittaker model of $\pi$, a Schwartz--Bruhat function $\phi$ in the space of a Weil representation, and
a section $f_s$ in the space of the representation parabolically induced (essentially) from $\tau\otimes|\det|^{s-1/2}$ to the double cover $\widetilde{\Sp}_{2n}$ of $\Sp_{2n}$ (the metaplectic group). The integral $\Psi^*(W,\phi,f_s)$ is obtained from $\Psi(W,\phi,f_s)$ by applying a normalized intertwining operator $\nintertwining{\tau}{s}$ to $f_s$. The equation defining $\gamma(s,\pi\times\tau,\psi)$ is essentially
\begin{align*}
\gamma(s,\pi\times\tau,\psi)\Psi(W,\phi,f_s)=\Psi^*(W,\phi,f_s).
\end{align*}
By the general theory of Rankin--Selberg integrals, one may always choose data such that $\Psi(W,\phi,f_s)$ is constant. The analytic behavior of $\gamma(s,\pi\times\tau,\psi)$ is therefore controlled by $\Psi^*(W,\phi,f_s)$. When the representations are supercuspidal (of any depth), and $l>n$, the integral is holomorphic, and therefore any pole must already be a pole of the intertwining operator. Of course, the poles of $\nintertwining{\tau}{s}$ are independent of $\pi$. It is therefore the zeros of the integral that we seek. For precise definitions see Section~\ref{subsection:The groups and the Shimura integral}.

In our setting $\pi=\pi_{\alpha}^{\omega}$ and $n=1$. The operator $\nintertwining{\tau}{s}$ is normalized such that it is holomorphic at $s=1$ unless $\tau$ is quadratic, and if it has a pole, it must be simple. The bulk part of the work is then to relate the existence of a zero to the relation between $\pi_{\alpha}^{\omega}$ and $\tau$.

The main idea in the proof is to choose the section $f_s$ such that
$\nintertwining{\tau}{s}f_s$ can be computed succinctly. The computation for the case of $p=2$ is more delicate. Even though the $\gamma$-factor is in this case entire, we are interested in obtaining its precise value. The precise normalization factor of the intertwining operator becomes crucial for this result, and we benefited from its computation by Sweet \cite{Sweet95} (reproduced in \cite[Appendix]{SzpGol15}, see also \cite{Dani,Szpruch11}).

Other works on Rankin--Selberg integrals for pairs of generic representations, one typically of a classical group, the other of a general linear group, include \cite{GPS,G,GRS4,BS,RGS,BS2}. For a survey on the Rankin--Selberg method see \cite{Bu}.

The first named author constructed the Langlands parameter for simple supercuspidal representations of odd orthogonal groups \cite{Adrian2016}. While the approach is similar, the computations in the symplectic case are far more subtle than the analogous computations in the odd orthogonal case. The main reason is that here, as opposed to the odd orthogonal case, the $\gamma$-factors may contain poles. In addition, the Shimura type integrals are in general more subtle to compute, because their construction involves a covering group and the Weil representation. Another difference is in the case $p=2$, which must be treated separately here.

Supercuspidal representations are a fundamental object of study in the category of complex smooth representations of reductive $p$-adic groups. In a certain sense, these are the building blocks of all irreducible representations. The construction of these representations in increasing generality has been the goal of a long line of works, including
\cite{Howe1977,Adler1998,BushnellKutzko1998,JKYu2001,Stevens2005,Stevens2008}. The depth of an irreducible representation was defined by Moy and Prasad \cite{MoyPrasad1994,MoyPrasad1996}, and supercuspidal representations of minimal nonzero depth are called simple supercuspidal, which as mentioned above were constructed recently in \cite{GR10,RY14}.

Under a mild assumption on the residual characteristic, Kaletha \cite{K13,K15} has constructed a correspondence for simple supercuspidal representations (in the context of tamely ramified connected reductive $p$-adic groups), and has verified that his correspondence satisfies many of the expected properties of the Langlands correspondence. The Langlands correspondence for simple supercuspidal representations was also studied in \cite{GR10,RY14}. Other recent works on simple supercuspidal representations and their Langlands parameters include \cite{O18} and \cite{FRT18}.

The rest of this work is organized as follows. Section~\ref{section:preliminaries} contains preliminaries. In Section~\ref{subsection:The groups and the Shimura integral} we review the construction of the Shimura type integrals for $\Sp_{2l}\times\GL_n$. The construction of simple supercuspidal representations of $\Sp_{2l}$ is described in Section~\ref{simplesupercuspidaldefinition}. The computation of the $\gamma$-factor is contained in Section~\ref{section:The computation of the gamma-factor}: Theorem~\ref{thm:main} is proved in Section~\ref{pneq2section}, the Langlands parameter is computed in Section~\ref{fullparameter}, and Theorem~\ref{thm:main2} is proved in Section~\ref{Q2}.

\subsection*{Acknowledgements}
We would like to thank Baiying Liu, Gordan Savin, Shaun Stevens and Geo Kam-Fai Tam for helpful conversations. Support to Adrian was provided by a grant from the Simons Foundation \#422638 and by a PSC-CUNY award, jointly funded by the
Professional Staff Congress and The City University of New York. Kaplan was
supported by the Israel Science Foundation, grant number 421/17. Lastly, we thank the referee for her/his careful reading of the manuscript and helpful remarks, which helped improve the presentation.

\section{preliminaries}\label{section:preliminaries}
\subsection{The groups and the Shimura type integral}\label{subsection:The groups and the Shimura integral}
Let $F$ be a $p$-adic field with a ring of integers $\mathfrak{o}$, maximal ideal $\mathfrak{p}=\mathfrak{p}_F$ and
a uniformizer $\varpi$. Denote the residue field $\mathfrak{o} / \mathfrak{p}$ by $\kappa_F$ and let $q$ be the cardinality of $\kappa_F$ (then $|\varpi|=q^{-1}$). Let $\psi$ be a non-trivial additive character of $F$. Denote the Weil index of $x\mapsto\psi(x^2)$ by $\gamma(\psi)$, and $\gamma_{\psi}(a)=\gamma(\psi_a)/\gamma(\psi)$ is the Weil factor,
where $\psi_a(x)=\psi(ax)$ for $a\in F^*$ (see \cite[Appendix]{Rao}).

Let $\Sp_{2l}=\Sp_{2l}(F)$ be the symplectic group on $2l$ variables, defined with respect to the symplectic
bilinear form $(x,y)\mapsto {}^tx\left(\begin{smallmatrix}&J_l\\-J_l\end{smallmatrix}\right)y$, where ${}^tx$ is the transpose of $x$ and
$J_l$ is the $l\times l$ matrix having $1$ on the anti-diagonal and $0$ elsewhere. Fix the Borel subgroup $B_l=T_l\ltimes N_l$ of upper triangular matrices in $\Sp_{2l}$, where $T_l$ is the torus and
$N_l$ is the unipotent radical. For $1\leq k\leq l$, let $Q_k<\Sp_{2l}$ be the standard maximal parabolic
subgroup whose Levi part is isomorphic to $\GL_k\times\Sp_{2(l-k)}$, and $\delta_{Q_k}$ be its modulus
character. For any $a\in \GL_k$, put $a^*=J_k(\transpose{a}^{-1})J_k$. The Levi subgroup of $Q_k$ is then
$\{\diag(a,h,a^*):a\in\GL_k,h\in\Sp_{2(l-k)}\}$. Also put $w_{l}=\left(\begin{smallmatrix}&I_{l}\\-I_{l}\end{smallmatrix}\right)\in \Sp_{2l}$ and for any $x,y\in\Sp_{2l}$, ${}^xy=x^{-1}yx$.

Let $\widetilde{\Sp}_{2l}$ denote the metaplectic group, i.e., the double cover of $\Sp_{2l}$. We write its elements as pairs
$\langle g,\epsilon\rangle$ where $g\in\Sp_{2l}$ and $\epsilon=\pm1$. The action is realized
using the cocycle of Rao \cite{Rao}. We identify $\Sp_{2l}$ with its image under the
map $g\mapsto\langle g,1\rangle$ (this is not a homomorphism). Also if $Q<\Sp_{2l}$, $\widetilde{Q}$ denotes its inverse image in $\widetilde{\Sp}_{2l}$.

For $l=1$, the cocycle of Rao coincides with the cocycle of Kubota \cite{Kubota}:
\begin{align}\label{eq:Kubota formula}
\sigma(g,g')=\Big(\frac{\mathbf{x}(gg')}{\mathbf{x}(g)},\frac{\mathbf{x}(gg')}{\mathbf{x}(g')}\Big),
\qquad\mathbf{x}\left(\begin{smallmatrix}a&b\\c&d\end{smallmatrix}\right)=\begin{cases}c&c\ne0,\\d&c=0.\end{cases}
\end{align}
Here $(\cdot , \cdot)$ is the quadratic Hilbert symbol.

Consider the parabolic subgroup $Q_n$ of $\Sp_{2n}$. If we identify $\GL_n$ with the Levi part of $Q_n$ by $a\mapsto\diag(a,a^*)$, the covering
$\widetilde{\GL}_n$ we obtain is simple, in the sense that the cocycle is given by the quadratic Hilbert symbol. Consequently, its genuine representations are in bijection with the representations of $\GL_n$, which can be made explicit by fixing a Weil symbol. For any representation $\tau$ of $\GL_n$, the genuine representation
$\glnrep\otimes\gamma_{\psi}$ of $\widetilde{\GL}_n$ is defined by
$\glnrep\otimes\gamma_{\psi}(\langle a,\epsilon\rangle )=\epsilon\gamma_{\psi}(\det{a})\glnrep(a)$, where $a\in\GL_n$ and $\epsilon=\pm1$.

Let $H_{r}$ be the $(2r+1)$-dimensional Heisenberg group, realized as the group of matrices
\begin{align*}
\left\{(x,y,z)=\left(\begin{array}{cccc}1&x&y&z\\&I_r&&y'\\&&I_r&x'\\&&&1\end{array}\right)\in \Sp_{2(r+1)}\right\}.
\end{align*}
The subgroup obtained by taking $y=z=0$ (resp., $x=z=0$) is denoted $X_r$ (resp., $Y_r$).

As mentioned in the introduction, our main tool is the computation of Shimura type integrals for $\Sp_{2l}\times\GL_n$. These integrals were developed in \cite{GRS4}. The second named author defined the corresponding local $\gamma$-factors and established their fundamental properties in \cite{Kaplan2015}. For the construction in \cite{Kaplan2015}, the representation $\pi$ of
$\Sp_{2l}$ was assumed to afford a Whittaker model with respect to the character of $N_l$ given by
$u\mapsto \psi^{-1}(\sum_{i=1}^{l-1}u_{i,i+1}-\half u_{l,l+1})$. This was compatible with \cite{ILM}. Here
$\pi$ will afford a Whittaker model with respect to the character $u\mapsto \psi^{-1}(\sum_{i=1}^{l}u_{i,i+1})$, which was used in \cite[\S~4]{GRS4}. The changes between the versions are minor and were (partially) described in \cite[Remark~4.4]{Kaplan2015}. To be explicit, and because we correct certain typos from \cite{GRS4,RGS,Kaplan2015}, we provide a complete definition of the integrals, even though we will only be using the
$\Sp_{2l}\times\GL_1$ integrals, with $l>1$.

Let $\mathcal{S}(F^r)$  be the space of Schwartz--Bruhat functions on
the row space $F^r$. Define the Fourier transform $\widehat{\phi}$
of $\phi\in \mathcal{S}(F^r)$ by
\begin{align*}
&\widehat{\phi}(y)=\int_{F^r}\phi(x)\psi(2xJ_r\transpose{y})dx,
\end{align*}
where $y$ and $x$ are rows. The measure on $F^r$ is the product of Haar measures on $F$, each self dual with respect to $\psi_2$, then $\widehat{\widehat{\phi}}(y)=\phi(-y)$ (see e.g., \cite[\S~24]{BH06}).  We emphasize here that the only times we will use Haar measures self-dual with respect to $\psi_2$ are when dealing with the Fourier transform.  Otherwise, all other Haar measures are self-dual with respect to $\psi$ in this paper.

Let $\omega_{\psi}$ denote the Weil representation of $\widetilde{\Sp}_{2r}\ltimes H_r$, realized
on $\mathcal{S}(F^r)$. It satisfies the following formulas (see \cite{Rallis82} and \cite[Chap.~2]{BerndtSchmidt1998}\footnote{Note the typo in \cite[\S~6, (1.3)]{GRS4}, and also in \cite[(1.4)]{RGS} and \cite[\S~3.3.1]{Kaplan2015}, in their definition of the action of
$\omega_{\psi}(\langle diag(a,a^*),\epsilon\rangle)$: $\gamma_{\psi}$ should be $\gamma_{\psi}^{-1}$ as in \cite{Rallis82} or \cite[Chap.~2]{BerndtSchmidt1998}.}):
for $\phi\in \mathcal{S}(F^r)$,
\begin{align}\nonumber
&\omega_{\psi}((x,0,z))\phi(\xi)=\psi(z)\phi(\xi+x),\\\nonumber
&\omega_{\psi}((0,y,0))\phi(\xi)=\psi(2\xi J_{r}\transpose{y})\phi(\xi),\\
\label{eq:Weil rep formula diag a}
&\omega_{\psi}(\langle diag(a,a^*),\epsilon\rangle)\phi(\xi)=\epsilon\gamma_{\psi}^{-1}(\det{a})\absdet{a}^{\half}\phi(\xi a),\\
\label{eq:Weil rep formula on u}
&\omega_{\psi}(\langle \left(\begin{smallmatrix}I_{r}&u\\&I_{r}\end{smallmatrix}\right),\epsilon\rangle)\phi(\xi)=\epsilon\psi(\xi J_{r}\transpose{u}\transpose{\xi})\phi(\xi),\\ \label{eq:Weil rep formula Fourier}
&\widehat{\phi}(\xi)=\beta_{\psi}\omega_{\psi}(w_r)\phi(\xi).
\end{align}
Here $\beta_{\psi}$ is a certain fixed eighth root of unity.

Let $\pi$ and $\tau$ be a pair of generic representations of
$\Sp_{2l}$ and $\GL_{n}$ ($l,n\geq1$). Assume that $\pi$ is realized in its Whittaker model $\whittaker{\pi}{\psi^{-1}}$, where
$\psi(u)=\psi(\sum_{i=1}^{l}u_{i,i+1})$ ($u\in N_l$).
For $s\in\mathbb{C}$, consider the space $V(\tau,s)$ of the normalized induced representation
\begin{align*}
\mathrm{Ind}_{\widetilde{Q}_n}^{\widetilde{\Sp}_{2n}}((\tau\otimes\gamma_{\psi})|\det|^{s-1/2}).
\end{align*}
Here we use $\gamma_{\psi}$, instead of $\gamma_{\psi}^{-1}$ as in \cite{GRS4,Kaplan2015}, this is implied by the correction to
\eqref{eq:Weil rep formula diag a} (to keep the formulas for the symplectic and metaplectic cases of \cite{Kaplan2015} as stated there, one simply has to regard $\gamma_{\psi}$ as the inverse of the Weil factor). The elements of this space are regarded as smooth genuine functions on $\wt{\Sp}_{2n} \times \GL_n$, where for any $g \in \wt{\Sp}_{2n}$
and $b\in\GL_n$,
\begin{align*}
f_s(\left(\begin{smallmatrix}b&u\\&b^*\end{smallmatrix}\right)g, a) = \gamma_{\psi}(\det b) \delta_{Q_n}^{1/2}\left(\begin{smallmatrix}b&\\&b^*\end{smallmatrix}\right)|\det b|^{s-1/2} f_s(g,ab),
\end{align*}
and the mapping $a \mapsto f_s(g,a)$ belongs to $\mathcal{W}(\tau,\psi)$, where $\psi$ is the character
of the group of upper triangular unipotent matrices in $\GL_n$ given by
$z\mapsto \psi(\sum_{i=1}^{n-1}z_{i,i+1})$.

Let $W\in\whittaker{\pi}{\psi^{-1}}$, $\phi\in\mathcal{S}(F^{\min(l,n)})$ and
$f_s\in V(\tau,s)$. The Shimura type integral for $\pi\times\tau$ is defined by
\begin{align*}
\Psi(W,\phi,f_s)=\begin{dcases}\int_{\lmodulo{N_l}{Sp_{2{l}}}}\int_{\lmodulo{Y_{{l}}}{H_{{l}}}}
\int_{R_{{l},{n}}}W(g)\omega_{\psi}(hg)\phi(\xi_0)f_s(\gamma_{l,n} r \heisenembedding{h}g,I_{{n}})\,dr\,dh\,dg&{l}<{n},\\
\int_{\lmodulo{N_l}{Sp_{2{l}}}}W(g)\omega_{\psi}(g)\phi(\xi_0)f_s(g,I_{{l}})\,dg&{l}={n},\\
\int_{\lmodulo{N_n}{Sp_{2{n}}}}\int_{R^{{l},n}}\int_{X_{{n}}}W(\rconj{\omega_{{l}-{n},{n}}}(r\heisenembedding{x}g))\omega_{\psi}(g)\phi(x)f_s(g,I_{{n}})\,dx\,dr\,dg&{l}>{n}.
\end{dcases}
\end{align*}
Here
\begin{align*}
&R_{{l},{n}}
=\left\{\left(\begin{array}{cccccc}I_{{n}-{l}-1}&0&r_1&0&r_2&r_3\\
&1&&&&r_2'\\&&I_{{l}}&&&0\\&&&I_{{l}}&&r_1'\\&&&&1&0\\&&&&&I_{{n}-{l}-1}\\\end{array}\right)\in \Sp_{2{n}}\right\}\qquad({l}\geq0),\\
&\xi_0=(0,\ldots,0,1)\in F^{{l}}\qquad\text{(a row vector),}\\
&\gamma_{{l},{n}}=\left(\begin{array}{cccc}&2I_{{l}}\\&&&-I_{{n}-{l}}\\I_{{n}-{l}}\\&&\half I_{{l}}\end{array}\right)\qquad({n}\geq{l}\geq0),\\
&R^{{l},{n}}
=\left\{\left(\begin{array}{cccccc}I_{{l}-{n}-1}&0&r_1\\
&1&&&&\\&&I_{{n}}&&&\\&&&I_{{n}}&&r_1'\\&&&&1&0\\&&&&&I_{{l}-{n}-1}\\\end{array}\right)\in
\Sp_{2{l}}\right\},\\
&\omega_{{l}-{n},{n}}=\left(\begin{array}{cccc}&I_{{l}-{n}}\\I_{{n}}
\\&&&I_{{n}}\\&&I_{{l}-{n}}\end{array}\right).
\end{align*}
The integral is absolutely convergent for $\Real(s)\gg0$, and admits meromorphic continuation to $\C(q^{-s})$. Moreover, there is a choice of data for which it becomes a nonzero constant, for all $s$ (see \cite{GRS4}).

Set $m=\min({n},{l})$ and $m_0=\max({n},{l})$.
The integrals can be regarded as trilinear forms on
$\whittaker{\pi}{\psi^{-1}}\times
\mathcal{S}(F^{l})\times
V(\tau,s)$ satisfying the
following equivariance properties:
\begin{align}\label{eq:equivariance props}
\begin{cases}
\Psi(g\cdot W,(h\langle g,\epsilon\rangle)\cdot\phi,(vh\langle g,\epsilon\rangle)\cdot
f_s)=\psi^{-1}(v)\Psi(W,\phi,f_s)&{l}<{n},\\
\Psi(g\cdot W,\langle g,\epsilon\rangle\cdot\phi,\langle g,\epsilon\rangle\cdot f_s)=\Psi(W,\phi,f_s)&{l}={n},\\
\Psi((vhg)\cdot W,(h\langle g,\epsilon\rangle)\cdot\phi,\langle g,\epsilon\rangle\cdot
f_s)=\psi^{-1}(v)\Psi(W,\phi,f_s)&{l}>{n}.
\end{cases}
\end{align}
Here $g\in \Sp_{2m}$, $\epsilon=\pm1$, $h\in H_{m}$,
\begin{align*}
v=\left(\begin{smallmatrix}
z & u & v \\ & I_{2(m+1)} & u' \\ &  & z^* \end{smallmatrix}\right)<N_{m_0},\qquad
\psi(v)=\psi(\sum_{i=1}^{m_0-m-2}z_{i,i+1}+u_{m_0-m-1,1}).
\end{align*}
Let $M(\tau,s):V(\tau,s)\rightarrow V(\tau^*,1-s)$ be the standard intertwining operator, given by the meromorphic continuation of the following integral:
\begin{align*}
[\intertwining{\tau}{s}f_s](g,a)=\int_{U_{n}}f_s(w_{{n}}^{-1}ug,d_{{n}}a^*)du,
\end{align*}
where $g\in\widetilde{\Sp}_{2{n}}$, $a\in \GL_{{n}}$, and $d_{{n}}=diag(-1,1,\ldots,(-1)^{{n}})\in \GL_{{n}}$.
The measure $du$ is defined as follows. We fix a Haar measure $du$ on $F$, which is self dual with respect to $\psi$. Then we identify each root subgroup $U_{\alpha}$ of $U_n$ with $F$, and thereby define the measure $du$ on $U_{\alpha}$. To fix this identification, for the short (resp., long) roots $\alpha$ we map $u\in U_{\alpha}$ to $x\in F$, if the non-trivial coordinate of $u$ above (resp., on) the anti-diagonal is $x$. Now the measure $du$ in the formula for $\intertwining{\tau}{s}f_s$ is the product measure on $U_n$.

Note that by definition $V(\tau^*,1-s)$ is the space of the representation induced from $(\tau^*\otimes\gamma_{\psi})|\det|^{1/2-s}$, so that the genuine representation of $\widetilde{\GL}_n$ attached to $\tau^*$ is $\tau^*\otimes\gamma_{\psi}$. When $\tau$ is irreducible, $\tau^*$ is the representation contragredient to $\tau$.
The normalized intertwining operator is defined by $\nintertwining{\tau}{s}=C(s,\tau,\psi)\intertwining{\tau}{s}$, where
$C(s,\tau,\psi)$ is the analog of Shahidi's local coefficient, and is defined via the functional equation
\begin{align}\label{eq:shahidi func equation}
&\int_{U_n}f_s(d_{{n}}w_{{n}}u,I_{{n}})\psi(u_{{n},{n}+1})\,du
=C(s,\tau,\psi)\int_{U_n}[\intertwining{\tau}{s}f_s](d_{{n}}w_{{n}}u,I_{{n}})\psi(u_{{n},{n}+1})\,du.
\end{align}
Also let $c(s,l,\tau)=\tau(2I_n)^{-2}|2|^{-2n(s-\half)}$ if $l\geq n$, otherwise $c(s,l,\tau)=1$.
The only dependence on the measure in \eqref{eq:shahidi func equation} is through the definition of the measure for the intertwining operator, and a different choice would be reflected in a different constant $C(s,\tau,\psi)$.

The space of trilinear forms satisfying \eqref{eq:equivariance props} is, outside a finite set of values of $q^{-s}$, one-dimensional.
Hence the integrals $\Psi(W,\phi,f_s)$ and $\Psi(W,\phi,M^*(\tau,s)f_s)$ are proportional.
The $\gamma$-factor is defined by
\begin{align}\label{def:gamma}
\gamma(s,\pi\times\tau,\psi)=\pi(-I_{2l})^n\tau(-I_n)^l\gamma(s,\tau,\psi)c(s,l,\tau)\frac{\Psi(W,\phi,M^*(\tau,s)f_s)}
{\Psi(W,\phi,f_s)},
\end{align}
whenever $\Psi(W,\phi,f_s)$ is not identically $0$. For the minimal case of $l=0$ (i.e., $\Sp_{2l}$ is the trivial group),
we must define $\gamma(s,\pi\times\tau,\psi)=\gamma(s,\tau,\psi)$, essentially because the Langlands dual group of $\Sp_{2l}$ is $\SO_{2l+1}(\C)$. The $\gamma$-factor is independent of the choices of measures for the integrals and the intertwining operator (the latter choice is offset by $C(s,\tau,\psi)$).

\begin{rmk}\label{rmk:correction to Kaplan2015}
In \cite{Kaplan2015}, $\gamma(s,\pi\times\tau,\psi)$ was defined as the ratio
\begin{align*}
\Psi(W,\phi,M^*(\tau,s)f_s)/ \Psi(W,\phi,f_s),
\end{align*}
or
\begin{align*}
c(s,l,\tau)\Psi(W,\phi,M^*(\tau,s)f_s)/ \Psi(W,\phi,f_s)
\end{align*}
depending on the convention for the integrals (\cite[Remark~4.4]{Kaplan2015}). Then a normalized version $\Gamma(s,\pi\times\tau,\psi)$ was defined by
\begin{align*}
\Gamma(s,\pi\times\tau,\psi)=\pi(-I_{2l})^n\tau(-I_n)^l\gamma(s,\pi\times\tau,\psi).
\end{align*}
The purpose of this normalization was to simplify some of the multiplicative formulas and in particular, obtain equality with Shahidi's $\gamma$-factor (\cite[Corollary~1]{Kaplan2015}).

The case of $l=0$ for $\Sp_{2l}$ was overlooked in \cite{Kaplan2015} and in this case $\Gamma(s,\pi\times\tau,\psi)$ was defined to be trivial instead of $\gamma(s,\tau,\psi)$.  In the definition of $\Gamma(s,\pi\times\tau,\psi)$ \cite[p.~408]{Kaplan2015} in general, one must again multiply by $\gamma(s,\tau,\psi)$ as we do here in \eqref{def:gamma} (see also \cite{LR}).

This correction implies the following list of changes to \cite{Kaplan2015}:
\begin{itemize}
\item In the multiplicative formula (\textit{loc. cit.}, Theorem~1 (2)) when $\pi'$ is a representation of $\Sp_{0}$,
$\Gamma(s,\pi',\psi)=\gamma(s,\tau,\psi)$ appears on the right hand side of the identity.
\item We get the correct $L$-function when data are unramified (\textit{loc. cit.}, Theorem~1 (3)).
\item In the formula describing the dependence on $\psi$ (\textit{loc. cit.}, Theorem~1 (6)) $2l$ changes into $2l+1$, which is the expected dependence since $\Sp_{2l}$ naturally lifts to $\GL_{2l+1}$.
\item The minimal case (\textit{loc. cit.}, Theorem~1 (7)) is now $\gamma(s,\tau,\psi)$.
\end{itemize}
This modification only applies to $\Sp_{2l}$, and not to $\wt{\Sp}_{2l}$ (the dual group of $\wt{\Sp}_{2l}$ is taken to be $\GL_{2l}(\C)$).

In this work we defined the $\gamma$-factor directly as the normalized one \eqref{def:gamma}, so we do not use the convention $\Gamma(\cdots)$, but \eqref{def:gamma} should be regarded as the corrected version of $\Gamma(s,\pi\times\tau,\psi)$ from \cite{Kaplan2015}. In particular,
the $\gamma$-factor $\gamma(s,\pi\times\tau,\psi)$ from \eqref{def:gamma} is equal to Shahidi's $\gamma$-factor defined in \cite{Sh3}.
\end{rmk}

We have the following twisting formula: for $a \in F^{\times}$, if
$W(\pi,\psi_a^{-1})\ne0$,
\begin{align}\label{twisting}
\gamma(s,\pi\times\tau,\psi_a)= \tau(a I_n)^{2l+1} |a|^{(2l+1)n(s-1/2)}  \gamma(s,\pi\times\tau,\psi).
\end{align}
Indeed if $a$ is a square, this was proved in \cite[Theorem~1 (6)]{Kaplan2015} (with the correction described in the remark above); in general this follows because $\gamma(s,\pi\times\tau,\psi_a)$ is identical with Shahidi's corresponding $\gamma$-factor \cite[Corollary~1]{Kaplan2015}. Note that a direct verification of \eqref{twisting} when $l=m$ (and $a$ is not a square) appeared in \cite[Proposition~3.5]{Zhang2016}.

As mentioned above, in this work $l>1$ and $n=1$. Then $\tau$ is a quasi-character of $F^{\times}$ and the Shimura type integral for $\pi\times\tau$ is defined by
\begin{align}\label{int:Shimura n=1}
\Psi(W,\phi,f_s)=\int_{N_1\backslash Sp_{2}}\int_{R^{l,1}}\int_{X_{1}}W({}^{\omega_{l-1,1}}(rxg))\omega_{\psi}(g)\phi(x)f_s(g,1)\ dx\ dr\ dg,
\end{align}
with
\begin{align*}
&R^{l,1}
=\left\{\left(\begin{array}{cccccc}I_{l-2}&0&r_1\\
&1&&&&\\&&1&&&\\&&&1&&r_1'\\&&&&1&0\\&&&&&I_{l-2}\\\end{array}\right)\in
\Sp_{2l}\right\},\qquad\omega_{l-1,1}=\left(\begin{array}{cccc}&I_{l-1}\\1\\&&&1\\&&I_{l-1}\end{array}\right).
\end{align*}
For certain computations, we will need to use the precise value of the normalizing factor $C(s,\tau,\psi)$ of the intertwining operator.
This value was computed in \cite{Sweet95,Szpruch11,SzpGol15}, albeit with slightly different conventions. Specifically, in \cite{Szpruch11} the character appearing on both sides of \eqref{eq:shahidi func equation} was
$\psi^{-1}$, the representation $\tau$ of $\GL_1$ was identified with a genuine representation of $\wt{\GL}_1$ by
$\tau\otimes\gamma_{\psi^{-1}}$, and the intertwining operator was given by $\int_{U_{1}}f_s(w_{{1}}u)\,du$. Under these conventions
the normalizing factor $C(s, \tau, \psi)$ was shown in \cite{Szpruch11} to be equal to the principal value of
\begin{align*}
\int_{F^{\times}}\tau(u)|u|^s\gamma_{\psi}(u)^{-1}\psi(u)|u|^{-1}du.
\end{align*}
(The measure $du$ is self dual with respect to $\psi$.)
According to Sweet \cite{Sweet95} (see \cite[Appendix]{SzpGol15}), this integral equals
\begin{align*}
\gamma(\psi_{-1})\tau(-1)\frac{\gamma(2s-1,\tau^2,\psi_2)}{\gamma(s,\tau,\psi)}.
\end{align*}
Since
\begin{align*}
[\intertwining{\tau}{s}f_s](g,1)=\int_{U_{1}}f_s(w_{{1}}^{-1}ug,d_1)\,du
=\gamma_{\psi}(-1)^{-1}\int_{U_{1}}f_s(w_{{1}}ug,1)\,du
\end{align*}
and $\psi_{-1}=\psi^{-1}$,
\begin{align}\label{C factor}
C(s,\tau,\psi)=\gamma_{\psi}(-1)\gamma(\psi)\tau(-1)\frac{\gamma(2s-1,\tau^2,\psi_2^{-1})}{\gamma(s,\tau,\psi^{-1})}=
\gamma_{\psi}(-1)\gamma(\psi)\frac{\gamma(2s-1,\tau^2,\psi_2)}{\gamma(s,\tau,\psi)}.
\end{align}
In particular, we can rewrite \eqref{def:gamma} for $n=1$ in the form
\begin{align}\label{def2:gamma}
\gamma(s,\pi\times\tau,\psi)=\pi(-I_{2l})\tau(-1)^l c(s,l,\tau)  \gamma_{\psi}(-1)\gamma(\psi)\gamma(2s-1,\tau^2,\psi_2)\frac{\Psi(W,\phi,M(\tau,s)f_s)}
{\Psi(W,\phi,f_s)}.
\end{align}
We will also use the fact that $\gamma(2s-1,\tau^2,\psi)$ has a pole at $s=1$ if and only if $\tau$ is quadratic (a simple computation).

\subsection{The simple supercuspidal representations of $\Sp_{2l}$}\label{simplesupercuspidaldefinition}
The simple supercuspidal representations of $\Sp_{2l}$ are those with smallest nonzero depth $\frac{1}{2l}$.
We briefly recall their construction. Let $K=\Sp_{2l}(\mathfrak{o})$, $T_0 = T_l \cap K$, $Z$ be the (finite) center of $\Sp_{2l}$,
and $I$ be the standard Iwahori subgroup, which is the preimage in $K$ of the Borel subgroup $B_l(\kappa_F)$. Also define
the subgroup $I^+$ of $I$, by replacing $\mathfrak{o}^{\times}$ with $1+\mathfrak{p}$ on the diagonal.

Additionally denote
\begin{align*}
T(q) = \{ t \in T_0| t^q=t\},\qquad Z(q)=Z \cap T(q).
\end{align*}
Fix an additive character $\psi^*$ of $F$ whose level is $1$. For any $\alpha\in F^{\times}$,
consider the following character $\psi^{*,\alpha}$ of $N_{l}$, given by
$\psi^{*,{\alpha}}(u)=\psi^*(\sum_{i=1}^{l-1}u_{i,i+1}+\alpha u_{l,l+1})$.

Following the works of Gross and Reeder \cite{GR10}, and Reeder and Yu \cite{RY14},
the affine generic characters
$\chi:ZI^+\rightarrow\mathbb{C}^*$ are of the form
\begin{align}\label{agc}
\begin{split}
\chi(zk)=\chi_{\underline{t}}^{\omega}(zk)=\omega(z)\psi^*(\sum_{i=1}^l t_i k_{i,i+1} + t_{l+1} \varpi^{-1}k_{2l,1}) \qquad (z\in Z,\,k\in I^+),\end{split}
\end{align}
where $\omega$ is a character of $Z$ and $\underline{t}=(t_1, \ldots, t_{l+1})$ with $t_i\in\mathfrak{o}^{\times}$. Since the level of $\psi^*$ is $1$, replacing any $t_i$ with $t_i+\varpi$ will produce the same character
$\chi_{\underline{t}}^{\omega}$. Therefore, with a minor abuse of notation we can already take $t_i\in\kappa_F^{\times}$.

In fact, a complete set of representatives of the $T(q)$-orbits of affine generic characters of $ZI^+$ is given by
$\chi_{\underline{t}}^{\omega}$,
$\underline{t}=(1,\ldots, 1,\alpha,t)$ where $\alpha$ varies over the square classes in $\kappa_F^{\times}$ and $t\in\kappa_F^{\times}$. Denote such a
character by $\chi_{\alpha, t}^{\omega}$.
By \cite[Proposition 9.3]{GR10}, the simple supercuspidal representations of $\Sp_{2l}$ are thus parameterized by the choices of central character $\omega$, $t \in \kappa_F^{\times}$ and $\alpha \in \kappa_F^{\times} / (\kappa_F^{\times})^2$.  We denote this representation by $\pi^{\omega}_{\alpha, t}$,
\begin{align*}
\pi^{\omega}_{\alpha, t}=\mathrm{ind}_{ZI^+}^{\Sp_{2l}} \chi_{\alpha, t}^{\omega}.
\end{align*}
Here $\mathrm{ind}$ denotes the compact induction.

Instead of using the above parametrization with $\omega$, $\alpha$ and $t$, we set $t = 1$ and let the affine generic characters be parameterized by central characters $\omega$, $\alpha \in \kappa_F^{\times} / (\kappa_F^{\times})^2$ and the various choices of uniformizer $\varpi$ in $F$. Since we have already fixed an (arbitrary) uniformizer $\varpi$ at the beginning of Section~\ref{subsection:The groups and the Shimura integral}, given $(\omega,\alpha)$ we denote the corresponding affine generic character and simple supercuspidal representation by
$\chi_{\alpha}^{\omega}$ and $\pi_{\alpha}^{\omega}$.

It is easy to see that any representation $\pi_{\alpha}^{\omega}$ is $\psi^{*,\alpha}$-generic.
Moreover, the following is a Whittaker function in $W(\pi_{\alpha}^{\omega},\psi^{*,\alpha})$:
\begin{align*}
W_{\alpha, \omega}(g) = \left\{
\begin{array}{rll}
\psi^{*,\alpha}(u) \chi_{\alpha}^{\omega} (zk) & \text{if }  g = uzk \in N_l Z I^+,\\
0 &  \text{otherwise.}
\end{array} \right.
\end{align*}

\subsection{Computations on $\overline{B}_1$}
For the computation of the integrals, it will be convenient to rewrite the $dg$-integral over
the Borel subgroup $\overline{B}_1$ of lower triangular matrices in $\Sp_2$. In this section we work out several
auxiliary computations that will be used repeatedly below.

For $a,c\in F^{\times}$, put
\begin{align}\label{b}
b=\begin{pmatrix}         a &  \\       a^{-1}c & a^{-1} \end{pmatrix}.
\end{align}
Note that the assumption $c\ne0$ excludes a subset of $\overline{B}_1$ of zero measure, hence will not affect our computations.
Recall that we identify $\Sp_2$ with its image in $\wt{\Sp}_2$ under $g\mapsto \langle g,1\rangle$. Here
\begin{align}\label{eq:decomposing b in the cover}
\langle b,1\rangle = (a^{-1},c)\langle \begin{pmatrix}         a &  \\        & a^{-1} \end{pmatrix},1\rangle \langle
\begin{pmatrix}         1 &  \\       c & 1 \end{pmatrix}
,1\rangle .
\end{align}
\begin{lem}\label{lemma:computation weil}
Let $\varphi\in\mathcal{S}(F)$ and $b$ be given by \eqref{b}. Then
\begin{align*}
&\omega_{\psi}(b) \phi(x)=(a^{-1}, c) \beta_{\psi}^{-2}\gamma_{\psi}^{-1}(a)\gamma_{\psi}(-1)
\int_F\psi(2axy)\psi(-cy^2)\int_F{\phi}(z) \psi(-2yz)
\ dz\ dy.
\end{align*}
Here $dz$ and $dy$ are self dual with respect to $\psi_2$.
\end{lem}
\begin{proof}
Using \eqref{eq:Kubota formula} together with
\eqref{eq:Weil rep formula diag a}--\eqref{eq:Weil rep formula Fourier}
we see that
\begin{align*}
 & \omega_{\psi}(b)\phi(x) = \ (a^{-1}, c)
\gamma_{\psi}^{-1}(a) \omega_{\psi}(\begin{pmatrix}
1 & 0\\
c & 1
\end{pmatrix})  \phi(ax) \\
&= (a^{-1}, c)
\gamma_{\psi}^{-1}(a) \omega_{\psi}(\langle w_1,1 \rangle
\langle \begin{pmatrix}
1 & -c\\
0 & 1
\end{pmatrix},1\rangle
\langle w_1^{-1},1\rangle)  \phi(ax) \\
&= \beta_{\psi}^{-1} (a^{-1}, c)
\gamma_{\psi}^{-1}(a) (\omega_{\psi}(\langle\begin{pmatrix}
1 & -c\\
0 & 1
\end{pmatrix},1\rangle
\langle w_1^{-1},1\rangle) {\phi})^{\widehat{}} (ax) \\
&= \beta_{\psi}^{-1} (a^{-1}, c)
\int_F
\gamma_{\psi}^{-1}(a) \omega_{\psi}(\langle\begin{pmatrix}
1 & -c\\
0 & 1
\end{pmatrix},1\rangle
\langle w_1^{-1},1\rangle) {\phi}(y) \psi(2axy) d y \\
&=  (-1,-1) \cdot \beta_{\psi}^{-1} (a^{-1}, c)
\int_F \gamma_{\psi}^{-1}(a)
\psi(-cy^2)
\omega_{\psi}(\langle -I_2,1\rangle\langle w_1,1\rangle) {\phi}(y) \psi(2axy) d y  \\
&=  (-1,-1) \cdot \beta_{\psi}^{-2} \gamma_{\psi}^{-1}(a) \gamma_{\psi}^{-1}(-1)  (a^{-1}, c)
\int_F   \psi(2axy)
\psi(-cy^2)
\int_F
{\phi}(z) \psi(-2yz) d z  d y.
\end{align*}
Note that we used
$\langle -I_2w_1,1\rangle = (-1,-1)\langle -I_2,1\rangle\langle w_1,1\rangle$, one equality before the last.
\end{proof}

\begin{lem}\label{lemma:computation M tau on b}
Let $f_s\in V(\tau,s)$ and $b$ be given by \eqref{b}, with $a\in 1+\mathfrak{p}$ and $c\ne0$. Then
\begin{align*}
[M(\tau,s) f_s](b, 1) =  \mathrm{vol}(\mathfrak{o}^{\times})\int_{F^{\times}}  (-uc,a)\tau(u^{-1}) |u|^{\frac{1}{2}-s}  \gamma_{\psi}^{-1}(u^{-1} a) f_s \left(
\begin{pmatrix}
1 & 0\\
a^2 u^{-1} +c & 1
\end{pmatrix}, -1\right)\, d^{\times} u.
\end{align*}
Here $\mathrm{vol}(\mathfrak{o}^{\times})$ is the volume of $\mathfrak{o}^{\times}$ with respect to $du$, and $d^{\times}u$ is the multiplicative measure of $F^{\times}$ assigning the volume $1$ to $\mathfrak{o}^{\times}$.
If $c=0$, the same formula holds except that the sign $(-uc,a)$ is replaced by $(au,a)$.
\end{lem}
\begin{proof}
By definition, for any $g\in \widetilde{\Sp}_{2}$,
\begin{align*}
[M(\tau,s) f_s](g, 1)  = \int_{U_1} f_s(w_1^{-1} u g, -1) du.
\end{align*}
Let
$u =
\left(\begin{smallmatrix}
1 & u\\
 & 1
\end{smallmatrix}\right)$. First assume $c\ne0$. Then
\begin{align}\label{eq:decomp 1 sign}
\langle w_1^{-1},1\rangle\langle u,1 \rangle\langle b,1\rangle =
\left\langle \begin{pmatrix}
-a^{-1} c & -a^{-1}\\
a +uca^{-1} & ua^{-1}
\end{pmatrix}, (-ac, a + ua^{-1} c) \right\rangle
\end{align}
as long as $a + ua^{-1} c \neq 0$ (when $a + ua^{-1} c=0$, the sign is $(ua,uc)$). Since $a$ and $c$ are given, we may assume that this condition holds, since it excludes only one value of $u$ (i.e., $cu\ne-a^2$).
Then,
\begin{align*}
[M(\tau,s) f_s](b, 1)  = \int_{F} (-ac, a + ua^{-1} c) f_s \left(
\begin{pmatrix}
-a^{-1} c & -a^{-1}\\
a +uca^{-1} & ua^{-1}
\end{pmatrix}, -1\right) du.
\end{align*}
As above we may assume $u \neq 0$. Since $f_s$ is left-invariant by $U_1$, we may row reduce upwards to obtain
\begin{align*}
[M(\tau,s)f_s](b, 1)  = \int_{F^{\times}} (-ac, a + ua^{-1} c) f_s \left(
\begin{pmatrix}
u^{-1} a & 0\\
a + uca^{-1} & ua^{-1}
\end{pmatrix}, -1\right)\, du.
\end{align*}
Recall that $du$ is a Haar measure on $F$, not $F^{\times}$.  We change the additive measure $du$ to the multiplicative one $d^{\times}u$, then obtain
\begin{align*}
[M(\tau,s)f_s]&(b,1) \\= & \mathrm{vol}(\mathfrak{o}^{\times}) \int_{F^{\times}} (-ac, a + ua^{-1} c) (ua^{-1}, a^2 u^{-1} + c) \cdot (-ua^{-1}(a^2 u^{-1} + c), a + uca^{-1})  \\
& \ \times  f_s \left(
\langle\begin{pmatrix}
u^{-1} a & 0\\
0 & u a^{-1}
\end{pmatrix},1\rangle
\langle\begin{pmatrix}
1 & 0\\
a^2 u^{-1} +c & 1
\end{pmatrix},1\rangle
, -1\right) |u|\, d^{\times} u\\
= & \mathrm{vol}(\mathfrak{o}^{\times}) \int_{F^{\times}} (-ac, a + ua^{-1} c) (ua^{-1}, a^2 u^{-1} + c) \cdot (-ua^{-1}(a^2 u^{-1} + c), a + uca^{-1})\\
& \ \times  \tau(u^{-1}) |u|^{\frac{1}{2}-s}  \gamma_{\psi}(u^{-1} a) f_s \left(
\begin{pmatrix}
1 & 0\\
a^2 u^{-1} +c & 1
\end{pmatrix}, -1\right)\, d^{\times} u,
\end{align*}
since $\tau$ is tamely ramified and hence $\tau(a)=1$ for $a \in 1 + \mathfrak{p}$.  Note that
\begin{align*}(-ua^{-1}(a^2 u^{-1} + c), a + uca^{-1}) = (-a - uca^{-1}, a + uca^{-1}) = 1,
\end{align*}
 since $(z, -z) = 1$ for any $z \in F^{\times}$.  Moreover, using $(z,1-z)=1$, one computes that
 \begin{align*}(-ac, a + ua^{-1} c) (ua^{-1}, a^2 u^{-1} + c) = (-uc, a)(ua^{-1}, -1).\end{align*}
The required formula now follows using $\gamma_{\psi}(u^{-1} a) = \gamma_{\psi}^{-1}(u^{-1} a) (-1, u^{-1} a)$.

When $c=0$ we compute as above, the only difference is in the signs: now there is no sign in
\eqref{eq:decomp 1 sign}, and when we decompose $\left(\begin{smallmatrix}u^{-1}a&0\\a&ua^{-1}\end{smallmatrix}\right)$ in the computation of
$[M(\tau,s)f_s](b,1)$, the sign is $(au,a)(ua^{-1}, -1)$ (instead of $(-uc, a)(ua^{-1}, -1)$).
\end{proof}
\section{The computation of the $\gamma$-factor}\label{section:The computation of the gamma-factor}
We study the $\gamma$-factor using a direct computation of integrals.
We begin by defining the data for the computation. Fix an additive character
$\psi$ of $F$, whose level is $1$. As in Section~\ref{subsection:The groups and the Shimura integral}, this defines the character $\psi^{-1}$ of $N_l$.

Using the notation of Section~\ref{simplesupercuspidaldefinition}, put $\psi^*=\psi^{-1}$ and $\alpha=1$.
Throughout, we will set $\pi = \pi_1^{\omega}$ and $\chi = \chi_1^{\omega}$. Now $\pi$ admits the Whittaker model
$W(\pi,\psi^{-1})$, where $\psi(u)=\psi(\sum_{i=1}^{l}u_{i,i+1})$ ($u\in N_l$).
Our Whittaker function is defined by
\begin{align*}
W(g) = \left\{
\begin{array}{rll}
\psi^{-1}(u) \omega(z) \chi(k) & \text{if} & g = uzk \in N_l Z I^+, \\
0 &  & \text{otherwise.}
\end{array} \right.
\end{align*}

Recall from Section~\ref{subsection:The groups and the Shimura integral} that we defined the measure $du$, which is self dual with respect to
the character $\psi$. For a measurable subset of $F$, $\mathrm{vol}(\cdots)$ will denote its measure assigned by $du$. Since $\psi$ is now of level $1$, $\mathrm{vol}(\frak{o})=q^{\frac{1}{2}}$ (see e.g., \cite[\S~23]{BH06}).  Also fix the measure $d^{\times}u$ as in
Lemma~\ref{lemma:computation M tau on b}, then $\mathrm{vol}^{\times}(\mathfrak{o}^{\times})=1$, where $\mathrm{vol}^{\times}(\cdots)$ is the measure assigned by $d^{\times}u$.  For computations, the measure $dx$ (resp., $dr$) of $X_{1}$ (resp., $R^{l,1}$) appearing in the integral \eqref{int:Shimura n=1} will be the measure (resp., product of measures) $du$.

\subsection{The case when $p \neq 2$}\label{pneq2section}
In this section, we compute $\gamma(s, \pi \times \tau, \psi)$ when $\tau$ is quadratic and tamely ramified. This implies in particular that the restriction of $\tau$ to $\mathfrak{o}^{\times}$ coincides with either $\gamma_{\psi}$ or $(\varpi,\cdot)\gamma_{\psi}$.
Most of the computations do not require $\tau$ to be quadratic, but this condition will eventually be used.

\begin{rmk}
Since $p>2$, there are only two classes in $\kappa_F^{\times} / (\kappa_F^{\times})^2$. At the end of this section, we will see that the result for the non-trivial class can be easily deduced from the case $\alpha=1$.
\end{rmk}

The Whittaker function $W$ was defined in Section~\ref{section:The computation of the gamma-factor}. To define $f_s$, consider the subgroup
\begin{align*}
\mathcal{N}=\left\{\left(\begin{array}{cc}1+\mathfrak{p}&\mathfrak{p}\\\mathfrak{p}^2&1+\mathfrak{p}\end{array}\right)\in\Sp_2\right\}.
\end{align*}
Since $p>2$, according to \cite[Theorem~2]{Kubota2} the section $\langle x,\vartheta(x)\rangle$ is a splitting (a homomorphism) of $\Sp_{2}(\mathfrak{o})$ where
\begin{align*}
\vartheta\left(\begin{smallmatrix}a&b\\c&d\end{smallmatrix}\right)=
\begin{cases}1& c=0 \text{ or } |c|=1,\\
(c,d)&\text{otherwise.}
\end{cases}
\end{align*}
In particular it is a splitting of $\mathcal{N}$. However, because $1 + \mathfrak{p} \subset (F^{\times})^2$, $\vartheta$ is trivial on $\mathcal{N}$, thus the trivial section $g\mapsto\langle g,1\rangle$ is a homomorphism of $\mathcal{N}$.

We define $f_s$ to be the function supported on $\wt{B}_1\mathcal{N}$, such that
\begin{align*}
f_s(\langle\left(\begin{smallmatrix}b&u\\&b^{-1}\end{smallmatrix}\right),\epsilon\rangle \langle v,1\rangle, a)=\epsilon|b|^{s+1/2}\gamma_{\psi}(b)\tau(ba),\qquad v\in\mathcal{N}.
\end{align*}
Also let $\phi$ be the characteristic function of $\mathfrak{p}$.

\begin{lem}\label{lemma:computation Psi with p not 2}
$\Psi(W, \phi, f_s) =  \beta_{\psi}^{-2} \gamma_{\psi}^{-1}(-1)  \mathrm{vol}(\mathfrak{p})^{l} \mathrm{vol}( \mathfrak{o}) \mathrm{vol}^{\times}(1 + \mathfrak{p})\mathrm{vol}(\mathfrak{p}^2)$.
\end{lem}
\begin{proof}
Writing the $dg$-integration over $\overline{B}_1$,
\begin{align*}
\Psi(W,\phi,f_s)=\int_{\overline{B}_1}\int_{R^{l,1}} \int_{F}
W({}^{\omega_{l-1,1}}(rxg)) \omega_{\psi}(g) \phi(x) f_s(g,1) \delta_B(b)\,d x\, d r\, d b.
\end{align*}
The support of $W$ is contained in $N_l Z I^+$. Looking at
${}^{\omega_{l-1,1}}(rxg)$ with $b=\left(\begin{smallmatrix}a&0\\a^{-1}c&a^{-1}\end{smallmatrix}\right)\in \overline{B}_1$,
we see that $W$ vanishes unless $a\in 1+\mathfrak{p}$, $x\in\mathfrak{p}$, $c\in\mathfrak{p}$ and the coordinates of
$r_1$ belong in $\mathfrak{p}$. In particular ${}^{\omega_{l-1,1}}(rxg)\in N_lI^+$ (see the definition of $W$).
We may also assume $c\ne0$. Then
\begin{align}\label{eq:W p not 2}
W({}^{\omega_{l-1,1}}(rxg))=\psi^{-1}(a^{-1}c\varpi^{-1}).
\end{align}

Using our definition of $\phi$, the double integral from Lemma~\ref{lemma:computation weil} becomes
\begin{align*}
\int_F   \psi(2axy)
\psi(-cy^2)
\int_{ \mathfrak{p}}
 \psi(-2yz)\ dz \ dy.
\end{align*}
This vanishes unless $y \in \mathfrak{o}$, then since
$a\in 1+\mathfrak{p}$ and $x,c\in\mathfrak{p}$,
\begin{align}\label{eq:phi p not 2}
&\omega_{\psi}(b) \phi(x)=\beta_{\psi}^{-2} \gamma_{\psi}^{-1}(-1)  \mathrm{vol}(\mathfrak{p}) \mathrm{vol}( \mathfrak{o}).
\end{align}
Note that $\gamma_{\psi}^{-1}(a)=1$ because $a\in 1+\mathfrak{p}\subset F^{\times2}$. Also observe that the measure assigned
by $dz$ or $dy$ (which are self dual with respect to $\psi_2$) is precisely the measure assigned by $du$ (justifying the usage of $\mathrm{vol}(\cdots)$ in \eqref{eq:phi p not 2}) because $|2|=1$.

Additionally for $a \in 1 + \mathfrak{p}$, $f_s(b,1)=0$ unless $c \in \mathfrak{p}^2$, in which case
by \eqref{eq:decomposing b in the cover},
\begin{align}\label{eq:f_s p not 2}
f_s(b, 1) = (a^{-1},c)\gamma_{\psi}(a)f_s(\langle\left(\begin{smallmatrix}     1 &  \\  c & 1 \end{smallmatrix}\right),1\rangle)=
f_s(\langle\left(\begin{smallmatrix}     1 &  \\  c & 1 \end{smallmatrix}\right),1\rangle)=1.
\end{align}
Combining \eqref{eq:W p not 2}--\eqref{eq:f_s p not 2} gives the result.
\end{proof}

\begin{lem}\label{lemma:computation M(tau,s)f_s with p not 2}
Let $c\in\mathfrak{p}$, and write $c\in c_0\varpi+\mathfrak{p}^2$ with $c_0\in\mathfrak{o}$. Then
\begin{align*}
[M(\tau,s)f_s]\left(\left(\begin{smallmatrix}1\\c&1\end{smallmatrix}\right), 1 \right)=
\begin{dcases}
A(\tau,\psi,s)\tau(-1) (q-1)(L(2s-1,\tau^2)-1)&c_0=0,\\
q^{-s}\tau(c_0\varpi) \gamma_{\psi}^{-1}(-c_0\varpi)&c_0\in\mathfrak{o}^{\times}.
\end{dcases}
\end{align*}
Here
\begin{align*}
A(\tau,\psi,s)=
\begin{dcases}
 q^{-s}\gamma_{\psi}^{-1}(\varpi)  \tau(\varpi) &\tau|_{\mathfrak{o}^{\times}} \equiv (\varpi, \cdot) \gamma_{\psi}|_{\mathfrak{o}^{\times}},\\
 q^{-1/2}  &\tau|_{\mathfrak{o}^{\times}} \equiv \gamma_{\psi}|_{\mathfrak{o}^{\times}}.
\end{dcases}
\end{align*}
\end{lem}
\begin{proof}
First assume $c\in\mathfrak{p}^2$. Thus, $c$ can be ignored and we see that
the integrand vanishes unless $u^{-1}\in\mathfrak{p}^{2}$. Hence by Lemma~\ref{lemma:computation M tau on b},
\begin{align*}
&[M(\tau,s)f_s]\left(\left(\begin{smallmatrix}1&\\c&1\end{smallmatrix}\right), 1 \right)\\&
=\mathrm{vol}(\mathfrak{o}^{\times}) \sum_{\ell=2}^{\infty}q^{(1/2-s)\ell} \tau^{\ell}(\varpi) \gamma_{\psi}^{-1}(\varpi^{\ell}) \int_{\mathfrak{o}^{\times}}(\varpi^{\ell},u)\gamma_{\psi}^{-1}(u) \tau(u)  \tau(-1) \, d^{\times}u\\&
=  \tau(-1)\mathrm{vol}(\mathfrak{o}^{\times}) \displaystyle\sum_{\ell = 1}^{\infty} q^{(\frac{1}{2}-s) (2\ell+1)}  \tau^{2\ell+1}(\varpi) \gamma_{\psi}^{-1}(\varpi) \int_{\mathfrak{o}^{\times}}  \gamma_{\psi}^{-1}(u) (\varpi, u) \tau(u)\, d^{\times} u\\&\quad
+  \tau(-1) \mathrm{vol}(\mathfrak{o}^{\times}) \displaystyle\sum_{\ell = 1}^{\infty} q^{(\frac{1}{2}-s) (2\ell)}  \tau^{2\ell}(\varpi) \int_{\mathfrak{o}^{\times}}  \gamma_{\psi}^{-1}(u) \tau(u)\, d^{\times} u
\\& =\begin{dcases}
 \tau(-1) \mathrm{vol}(\mathfrak{o}^{\times}) \gamma_{\psi}^{-1}(\varpi) q^{\frac{1}{2}-s} \tau(\varpi) \left( \frac{1}{1 - q^{1-2s} \tau^2(\varpi)}  - 1 \right) & \tau|_{\mathfrak{o}^{\times}} \equiv (\varpi, \cdot) \gamma_{\psi}|_{\mathfrak{o}^{\times}},\\
\tau(-1)\mathrm{vol}(\mathfrak{o}^{\times}) \left( \frac{1}{1 - q^{1-2s} \tau^2(\varpi)}  - 1 \right) & \tau|_{\mathfrak{o}^{\times}} \equiv \gamma_{\psi}|_{\mathfrak{o}^{\times}}.
\end{dcases}
\end{align*}
Here we used the fact that $\gamma_{\psi}$ is a character of $\mathfrak{o}^{\times}$, which is non-trivial because the
level of $\psi$ is $1$ (as opposed to $0$).

Assume now that $c\in c_0\varpi+\mathfrak{p}^2$ for some $c_0\in\mathfrak{o}^{\times}$. We see that the integrand
vanishes unless $u^{-1}+c\in\mathfrak{p}^2$, so $u^{-1}=-u_0c_0\varpi$ with
$u_0\in1+\mathfrak{p}$ and we obtain
\begin{align*}
 \tau(-1) \mathrm{vol}(\mathfrak{o}^{\times}) q^{1/2-s}\int_{1+\mathfrak{p}}\tau(-u_0c_0\varpi ) \gamma_{\psi}^{-1}(-u_0c_0\varpi )
\, d^{\times} u_0.
\end{align*}
Since $\tau$ and $\gamma_{\psi}$ are trivial on $1+\mathfrak{p}$, $(1+\mathfrak{p},\cdot)=1$, $\mathrm{vol}(\mathfrak{o}^{\times})=(q-1)q^{-1/2}$ and
$\mathrm{vol}^{\times}(1+\mathfrak{p})=(q-1)^{-1}$, the result follows.
\end{proof}

\begin{lem}\label{lemma:computation Psi M(tau,s)f_s with p not 2}
\begin{align*}
&\Psi(W, \phi, M(\tau,s)f_s)=\Psi(W, \phi, f_s)A(\tau,\psi,s)\tau(-1)(q-1)(L(2s-1,\tau^2)-1-\frac{\Lambda(\tau,\psi,s)}{q-1}).
\end{align*}
Here
\begin{align*}
\Lambda(\tau,\psi,s)=
\begin{dcases}
 1&\tau|_{\mathfrak{o}^{\times}} \equiv (\varpi, \cdot) \gamma_{\psi}|_{\mathfrak{o}^{\times}},\\
 -\tau(\varpi)\gamma_{\psi}(\varpi)q^{1/2-s}G(\psi^{-1})&\tau|_{\mathfrak{o}^{\times}} \equiv \gamma_{\psi}|_{\mathfrak{o}^{\times}},
\end{dcases}
\end{align*}
where $G(\psi^{-1})= \sum_{x \in \mathbb{F}_q^{\times}} \psi^{-1}(x) (\varpi, x)$ is a Gauss sum over the multiplicative group of the
finite field $\mathbb{F}_q$ with $q$ elements.
\end{lem}
\begin{proof}
We begin as in the proof of Lemma~\ref{lemma:computation Psi with p not 2}. In particular, $W$ vanishes unless
$a\in 1+\mathfrak{p}$, $x,c\in\mathfrak{p}$. We obtain
\begin{align*}
&\Psi(W, \phi, M(\tau,s)f_s)\\&=\beta_{\psi}^{-2} \gamma_{\psi}^{-1}(-1)\mathrm{vol}(\mathfrak{p})^{l} \mathrm{vol}( \mathfrak{o}) \mathrm{vol}^{\times}(1 + \mathfrak{p})
\\&\times\left( \int_{\mathfrak{p}^2} \psi^{-1}(c\varpi^{-1}) [M(\tau,s)f_s]\left(\left(\begin{smallmatrix}1&\\c&1\end{smallmatrix}\right), 1 \right)\, dc + \int_{\mathfrak{p} \setminus \mathfrak{p}^2} \psi^{-1}(c\varpi^{-1})  [M(\tau,s)f_s]\left(\left(\begin{smallmatrix}1&\\c&1\end{smallmatrix}\right), 1 \right)\, dc \right).
\end{align*}
Now we apply Lemma~\ref{lemma:computation M(tau,s)f_s with p not 2}. Since $\psi(c\varpi^{-1})$ is trivial when $c\in\mathfrak{p}^2$,
the first integral becomes
\begin{align}\label{eq:first sum}
\mathrm{vol}(\mathfrak{p}^2)A(\tau,\psi,s)\tau(-1)(q-1)(L(2s-1,\tau^2)-1).
\end{align}
For the second integral, because $\psi^{-1}(c\varpi^{-1})=\psi^{-1}(c_0)$ (where we have written $c \in c_0\varpi+\mathfrak{p}^2$ with $c_0\in\mathfrak{o}^{\times}$), it equals
\begin{align*}
&q^{-s}\int_{\mathfrak{p}\setminus\mathfrak{p}^2} \psi^{-1}(c\varpi^{-1}) \tau(c_0\varpi)\gamma_{\psi}^{-1}(-c_0\varpi)\ dc \nonumber \\\nonumber
&=q^{-s}\int_{\mathfrak{p} \setminus \mathfrak{p}^2}  \psi^{-1}(c_0\varpi^{-1})  \tau(c_0\varpi )\gamma_{\psi}^{-1}(-c_0\varpi ) \, dc\\
&=q^{-s}\tau(\varpi)\gamma_{\psi}^{-1}(-\varpi)\int_{\mathfrak{p} \setminus \mathfrak{p}^2}  \psi^{-1}(c_0\varpi^{-1})  \tau(c_0)\gamma_{\psi}^{-1}(c_0)(c_0,-\varpi) \, dc.
\end{align*}
Using the integration formula
\begin{align*}
\int_{\mathfrak{p}\setminus\mathfrak{p}^2} \xi(c)\,dc
=q^{-1}\int_{\mathfrak{o}^{\times}} \xi(v\varpi)\,dv=q^{-1}(q^{1/2}-q^{-1/2})
\int_{\mathfrak{o}^{\times}} \xi(v\varpi)\,d^{\times}v,
\end{align*}
and the fact that the Hilbert symbol is trivial on $\mathfrak{o}^{\times} \times \mathfrak{o}^{\times}$ (since $p$ is odd),
we obtain
\begin{align}\label{eq:second sum 2}
&q^{-s}\tau(\varpi)\gamma_{\psi}^{-1}(-\varpi)\mathrm{vol}(\mathfrak{p}^2)(q-1)\int_{\mathfrak{o}^{\times}}  \psi^{-1}(v)  \tau(v)\gamma_{\psi}^{-1}(v)(\varpi,v) \, d^{\times}v.
\end{align}

Now if $\tau|_{\mathfrak{o}^{\times}} \equiv (\varpi, \cdot) \gamma_{\psi}|_{\mathfrak{o}^{\times}}$, the last $d^{\times}v$-integral
becomes $(1-q)^{-1}$. Combining \eqref{eq:first sum} with \eqref{eq:second sum 2}, and recalling that in this case
$A(\tau,\psi,s)=q^{-s}\tau(\varpi)\gamma_{\psi}^{-1}(\varpi)$, the sum of integrals
in the expression for $\Psi(W, \phi, M(\tau,s)f_s)$ equals
\begin{align*}
\mathrm{vol}(\mathfrak{p}^2)A(\tau,\psi,s)\tau(-1)(q-1)[
 L(2s-1,\tau^2)-1-\frac{1}{q-1}\tau(-1) \gamma_{\psi}^{-1}(-1) (-1, \varpi)].
\end{align*}
Then our assumption on $\tau|_{\mathfrak{o}^{\times}}$ and Lemma~\ref{lemma:computation Psi with p not 2} imply
\begin{align*}
&\Psi(W, \phi, M(\tau,s)f_s)=\Psi(W, \phi, f_s)A(\tau,\psi,s)\tau(-1)(q-1)(L(2s-1,\tau^2)-1-\frac{1}{q-1}).
\end{align*}

For the case $\tau|_{\mathfrak{o}^{\times}} \equiv \gamma_{\psi}|_{\mathfrak{o}^{\times}}$, we observe that the
integral in \eqref{eq:second sum 2} is invariant under $1 + \mathfrak{p}$, hence \eqref{eq:second sum 2} equals
\begin{align*}
&q^{-s}\tau(\varpi)\gamma_{\psi}^{-1}(-\varpi)\mathrm{vol}(\mathfrak{p}^2)
\int_{\mathfrak{o}^{\times} / 1 + \mathfrak{p}} \psi^{-1}(w) (\varpi, w) \,d^{\times} w
=q^{-s}\tau(\varpi)\gamma_{\psi}^{-1}(-\varpi) \mathrm{vol}(\mathfrak{p}^2)G(\psi^{-1}).
\end{align*}
Together with \eqref{eq:first sum} and since now $A(\tau,\psi,s)=q^{-1/2}$, the sum of two integrals equals
\begin{align*}
\mathrm{vol}(\mathfrak{p}^2)A(\tau,\psi,s)\tau(-1)(q-1)[
 L(2s-1,\tau^2)-1+ \frac{1}{q-1}\tau(-\varpi) \gamma_{\psi}^{-1}(-\varpi) q^{1/2-s} G(\psi^{-1})].
\end{align*}
As above, the result now follows from Lemma~\ref{lemma:computation Psi with p not 2}.
\end{proof}

We can now deduce Theorem~\ref{thm:main} for the case $\alpha=1$, i.e., $\alpha$ is a representative of the trivial square class
in $\kappa_F^{\times} / (\kappa_F^{\times})^2$:
\begin{thm}\label{thm:main p not 2}
Let $\tau$ be a tamely ramified quadratic character of $F^{\times}$, and recall that $\pi = \pi_1^{\omega}$.  Then $\gamma(s, \pi \times \tau, \psi)$ has a pole at $s = 1$
if and only if $\tau$ is the unique non-trivial quadratic character of $F^{\times}$ such that $\tau|_{\mathfrak{o}^{\times}} \equiv \gamma_{\psi}|_{\mathfrak{o}^{\times}}$ and $\tau(\varpi) =  \gamma_{\psi}^{-1}(\varpi)  \frac{|G(\psi^{-1})|}{G(\psi^{-1})}$.
\end{thm}
\begin{proof}
By Lemma~\ref{lemma:computation Psi with p not 2}, $\Psi(W,\phi,f_s)$ is in particular nonvanishing, hence the zeros and poles of
$\gamma(s, \pi \times \tau, \psi)$ may be read off the numerator of \eqref{def2:gamma}.
Since $\tau^2=1$, $\gamma(2s-1,\tau^2,\psi_2)$ has a simple pole at $s=1$.  Also note that $L(1,\tau^2)=(1-q^{-1})^{-1}$.

Assume the restriction of $\tau$ to $\mathfrak{o}^{\times}$ coincides with $(\varpi,\cdot)\gamma_{\psi}$.
Then $\Lambda(\tau,\psi,s)=1$, hence by Lemma~\ref{lemma:computation Psi M(tau,s)f_s with p not 2},
$\Psi(W,\phi,M(\tau,s)f_s)$ has a simple zero at $s=1$. This zero cancels with the pole of $\gamma(2s-1,\tau^2,\psi_2)$, and we deduce that $\gamma(s, \pi \times \tau, \psi)$ does not have a pole at $s=1$.

Now consider the case $\tau|_{\mathfrak{o}^{\times}} \equiv \gamma_{\psi}|_{\mathfrak{o}^{\times}}$.
Then by Lemma~\ref{lemma:computation Psi M(tau,s)f_s with p not 2},
$\Psi(W,\phi,M(\tau,s)f_s)$ has a zero at $s=1$ if and only if
\begin{align*}
\tau(\varpi)=-q^{1/2}\gamma_{\psi}^{-1}(\varpi)G(\psi^{-1})^{-1}.
\end{align*}
Since $G(\psi^{-1})^2 = (-1, \varpi) q$ (\cite[(23.6.3)]{BH06}), $G(\psi^{-1}) = \zeta q^{1/2}$ for some fourth root of unity $\zeta$, and $G(\psi^{-1}) / |G(\psi^{-1})| = \zeta$. Therefore, we obtain
\begin{align*}
\tau(\varpi)=-\gamma_{\psi}^{-1}(\varpi)\frac{|G(\psi^{-1})|}{G(\psi^{-1})}.
\end{align*}
When $\tau(\varpi)$ satisfies this condition, the pole of $\gamma(1,\tau^2,\psi_2)$ is canceled. The remaining option is
$\tau(\varpi)=\gamma_{\psi}^{-1}(\varpi)\frac{|G(\psi^{-1})|}{G(\psi^{-1})}$ (because $\tau^2(\varpi)=1$), and in this case
$\Psi(W,\phi,M(\tau,s)f_s)$ does not have a zero, so that
$\gamma(s, \pi \times \tau, \psi)$ has a pole at $s=1$.
\end{proof}
\begin{rmk}
Note that for
$\tau(\varpi)=\pm\gamma_{\psi}^{-1}(\varpi)\frac{|G(\psi^{-1})|}{G(\psi^{-1})}$, indeed we have $\tau^2(\varpi) = 1$ (see \cite[Lemma 3.4]{Dani}).
\end{rmk}

Thus far we have computed $\gamma(s, \pi \times \tau, \psi)$ for $\pi=\pi_{1}^{\omega}$. In general, recall the representation $\pi_{\alpha}^{\omega}$, where $\alpha$ is a non-trivial class in $\kappa_F^{\times} / (\kappa_F^{\times})^2$, constructed in Section~\ref{simplesupercuspidaldefinition}. This representation is $\psi^{*,\alpha}$-generic, but $\psi^{*,\alpha}$ is conjugate to the character $u\mapsto \psi(\alpha\sum_{i=1}^{l}u_{i,i+1})$ of $N_l$, via the image of $\diag(\alpha^{l-1},  \ldots,\alpha, 1)\in\GL_l$ in $T_l$. Therefore, $\gamma(s, \pi_{\alpha}^{\omega}\times\tau,\psi_{\alpha})$ is defined.

Repeating all of the above computations with $\psi_{\alpha}$ instead of $\psi$, the only dependency on $\psi_{\alpha}$ is that
it must have the same level as $\psi$. In particular, we need $\psi_{\alpha}(a^{-1}c\varpi^{-1})$ to be trivial precisely when
$\psi(a^{-1}c\varpi^{-1})$ is. Since $|\alpha|=1$, this assumption holds and we deduce that
$\gamma(s,\pi_{\alpha}^{\omega}\times\tau,\psi_{\alpha})$ has a pole at $s=1$ if and only if
$\tau$ is the unique non-trivial quadratic character of $F^{\times}$ such that $\tau(\varpi) =  \gamma_{\psi_{\alpha}}^{-1}(\varpi)  \frac{|G(\psi^{-1}_{\alpha})|}{G(\psi^{-1}_{\alpha})}$ and $\tau|_{\mathfrak{o}^{\times}} \equiv \gamma_{\psi_{\alpha}}|_{\mathfrak{o}^{\times}}$. This completes the proof of Theorem~\ref{thm:main}.

\subsection{The Langlands parameter}\label{fullparameter}
In this section we prove Theorem~\ref{thm:main param}.
In order to describe the Langlands parameter of a simple supercuspidal representation of $\Sp_{2l}$, we need the parameterization of simple supercuspidal representations of general linear groups. Such representations are parameterized by a central character $\omega$, a uniformizer, and a particular root of the evaluation of $\omega$ on the particular uniformizer. For the details, see \cite[\S~3.1]{AL14}.

Recall that $\pi = \pi^{\omega}_{\alpha}$ is a simple supercuspidal representation of $\Sp_{2l}$ induced from the character $\chi = \chi^{\omega}_{\alpha}$ of $ZI^+$.  Then $\chi$ is the restriction to $\Sp_{2l}$ of a unique character $\tilde\chi$ of $I^+_{\GL_{2l}}$, which is invariant under the involution $\theta$ defining $\Sp_{2l}$.  Here, $I^+_{\GL_{2l}}$ is the pro-unipotent radical of the standard Iwahori in $\GL_{2l}$.  Let $\varphi_{\pi}$ be the Langlands parameter of $\pi$. By \cite{BHS17}, $\varphi_{\pi}=\varphi_1 \oplus \varphi_2$, where $\varphi_1$ is an irreducible $2l$-dimensional representation of the Weil group $W_F$, which corresponds to a simple supercuspidal representation $\Pi_1$ of $\GL_{2l}$, whose affine generic character (see \cite{AL14}) is given on $I^+_{GL_{2l}}$ by $(\tilde{\chi})^2|_{I_{\GL_{2l}}^+}$, and $\varphi_2$ is a $1$-dimensional representation of $\GL_1$, which equals $\tau_{\alpha}$ by Theorem \ref{thm:main}.

The involution $\theta$ is given by
\begin{align*}
\theta(g) = \left(\begin{smallmatrix}&J_l\\-J_l\end{smallmatrix}\right)^{-1} ({}^t g^{-1}) \left(\begin{smallmatrix}&J_l\\-J_l\end{smallmatrix}\right).
\end{align*}
An affine generic character of $I_{\GL_{2l}}^+$ takes the form
\begin{align*}
\lambda_{\underline{s}}(g)=\psi^{-1}(\sum_{i=1}^{2l-1}s_i v_{i,i+1} + s_{2l} \varpi^{-1}v_{2l,1} ).
\end{align*}

One can check, by root group calculations, that
\begin{align*}
\lambda_{\underline{s}}(\theta(g)) = \psi^{-1}(-s_{2l-1} v_{1,2} -... - s_{l+1} v_{l-1,l} + s_l v_{l,l+1} - s_{l-1} v_{l+1,l+2} - ... - s_{1} v_{2l-1,2l} + s_{2l} \varpi^{-1}v_{2l,1} ).
\end{align*}
Then $\lambda_{\underline{s}}$ is $\theta$ invariant if and only if
\begin{align*}
s_{2l-1} = -s_1, \quad s_{2l-2} = -s_2,\quad\ldots, \quad s_{l+1} = -s_{l-1}.
\end{align*}
Assume that $\lambda_{\underline{s}}$ is $\theta$ invariant. Thus
\begin{align*}
\lambda_{\underline{s}}(g)=\psi^{-1}(s_1 v_{1,2} +... + s_{l-1} v_{l-1,l} + s_l v_{l,l+1}  -s_{l-1} v_{l+1,l+2} + ... -s_1 v_{2l-1,2l} + s_{2l} \varpi^{-1}v_{2l,1} ).
\end{align*}
Restricting this to $I^+_{\Sp_{2l}}$, we get (since $v_{2l-1,2l} = -v_{1,2}$, $v_{2l-2,2l-1} = -v_{2,3}$, etc., in $\Sp_{2l}$)
\begin{align*}
\lambda_{\underline{s}}|_{I^+_{\Sp_{2l}}}(g)=\psi^{-1}(2s_1 v_{1,2} +\ldots + 2s_{l-1} v_{l-1,l} + s_l v_{l,l+1}  + s_{2l} \varpi^{-1}v_{2l,1} ).
\end{align*}

Therefore, setting
\begin{align*}
s_1 = \ldots =s_{l-1}=1/2,\quad s_l = \alpha, \quad s_{2l} = 1,
\end{align*}
we obtain a character $\tilde{\chi} = \lambda_{\underline{s}}$, whose restriction to $I^+=I^+_{\Sp_{2l}}$ coincides with that of $\chi$. Moreover, $\tilde{\chi}$ is $\theta$-invariant.  We are interested in the square of $\tilde{\chi}$, which is given by the tuple
$(1,1, \ldots , 1,2\alpha, -1, -1, \ldots , -1, 2)$, where $2 \alpha$ is in the $l$-th position.  More precisely, $\tilde{\chi}^2$ is given  on $I^+_{\GL_{2l}}$ by
\begin{align*}
k \mapsto \psi(-k_{1,2} - k_{2,3} - \ldots- k_{l-1,l} - 2 \alpha k_{l,l+1} + k_{l+1,l+2} + k_{l+2,l+3} + \ldots + k_{2l-1,2l} - 2\varpi^{-1} k_{2l,1}).
\end{align*}
Following the description in \cite[\S~3.1]{AL14}, we can conjugate this character by an element of $T_0$, to get the tuple $(1, 1, \ldots, 1, 1, (-1)^{l+1} 4 \alpha)$, which yields the character
\begin{align*}
k \mapsto \psi(k_{1,2} + k_{2,3} + \ldots + k_{2l-1,2l} + \varpi_{\alpha,l}^{-1} k_{2l,1}),\qquad \varpi_{\alpha,l}^{-1}=(-1)^{l+1} 4 \alpha\varpi^{-1},
\end{align*}
of $I^+_{\GL_{2l}}$. We re-denote this character by $\tilde{\chi}^2$.  Since the uniformizer $\varpi$ for $\pi$ is already fixed, we conclude that the uniformizer for $\Pi_1$ is $\varpi_{\alpha,l}$.

Now assume $p \nmid 2l$. Then $\varphi_1 = \Ind_{W_E}^{W_F}(\xi)$ for some tamely ramified
extension $E/F$ of degree $2 l$, and some character $\xi$ of $E^{\times}$ (see \cite{BH06}). Here, we have used the isomorphism $W_E^{\mathrm{ab}} \cong E^{\times}$, and pulled $\xi$ back to $W_E$.  Our goal now is to describe $E$ and $\xi$.  Let $\zeta$ be a $2l$-th root of $\varpi_{\alpha,l}$, and set $E = F(\zeta)$.
This is a totally ramified extension of $F$, thus $\kappa_E=\kappa_F$. Then $\xi$ is a character of $E^{\times} = \langle \zeta \rangle \times \kappa_F^{\times} \times (1 + \mathfrak{p}_E)$.

Relative to the basis
\begin{align*}
\zeta^{2l-1}, \zeta^{2l-2}, \cdots, \zeta, 1
\end{align*}
of $E/F$, we have an embedding
\begin{align*}
\iota : E^{\times} \hookrightarrow \GL_n(F).
\end{align*}
By \cite{BHS17}, for all $x \in 1 + \mathfrak{p}_E$,
\begin{align*}
\xi(x) = \tilde{\chi}^2(\iota(x))= \psi(\sum_{i=1}^{2l-1}\iota(x)_{i,i+1} + \varpi_{\alpha,l}^{-1} \iota(x)_{2l,1}).
\end{align*}

Since the determinant of $\varphi_{\pi}$ is trivial (by \cite{CKPS}), the central character $\omega_{\Pi_1}$ of $\Pi_1$ must satisfy $\omega_{\Pi_1} = \tau_{\alpha}^{-1}|_{F^{\times}} = \tau_{\alpha}|_{F^{\times}}$. Therefore, since $\omega_{\Pi_1} = \mathrm{det}(\varphi_1) = \mathrm{det}(\mathrm{Ind}_{W_E}^{W_F} 1_E) \otimes \xi|_{F^{\times}}$ (see \cite{BH06}), we deduce that $\xi|_{\kappa_F^{\times}} = \tau_{\alpha}|_{\kappa_F^{\times}} \otimes (\mathrm{det}(\mathrm{Ind}_{W_E}^{W_F} 1_E))^{-1}$. Thus, we have determined $\xi|_{1 + \mathfrak{p}_E}$ and $\xi|_{\kappa_F^{\times}}$. It remains to find $\xi(\zeta)$.

According to \cite[Corollary 3.12, \S~3.4, \S~3.6]{AL14},
\begin{align*}
\xi(\zeta) = \delta \cdot \lambda_{E/F}(\psi_{\alpha})^{-1},
\end{align*}
where $\lambda_{E/F}(\psi_{\alpha})$ is the Langlands constant (\cite{La6}),
and $\delta$ is the coefficient of $q^{1/2-s}$ in $\gamma(s, \Pi_1, \psi_{\alpha})$.  By the Langlands correspondence for general linear groups, $\delta$ is precisely the coefficient of $q^{1/2-s}$ in $\gamma(s, \varphi_1,  \psi_{\alpha})$.
Now, since
\begin{align*}
\gamma(s, \pi, \psi_{\alpha})=\gamma(s, \varphi_{\pi}, \psi_{\alpha}) = \gamma(s, \varphi_1,  \psi_{\alpha}) \gamma(s, \varphi_2,  \psi_{\alpha})=\gamma(s, \Pi_1,  \psi_{\alpha}) \gamma(s, \tau_{\alpha},  \psi_{\alpha}),
\end{align*}
we can compute $\delta$ using $\gamma(s, \pi, \psi_{\alpha})=\gamma(s, \pi\times1, \psi_{\alpha})$ and $\gamma(s, \tau_{\alpha},  \psi_{\alpha})$.

\begin{thm}
$\gamma(s, \pi,  \psi_{\alpha}) = \omega(-I_{2l})
\gamma(\psi_{\alpha})^{-1}\gamma_{\psi_{\alpha}}^{-1}(\varpi)q^{1/2-s}$.
\end{thm}
\begin{proof}
By Lemma~\ref{lemma:computation Psi M(tau,s)f_s with p not 2}, \eqref{def2:gamma} and since
$\pi(-I_{2l})=\omega(-I_{2l})$,
\begin{align*}
& \gamma(s, \pi\times 1,  \psi_{\alpha})  \\&=
\omega(-I_{2l})c(s,l,\tau)
\gamma_{\psi_{\alpha}}(-1)\gamma(\psi_{\alpha})\gamma(2s-1,\tau^2,(\psi_{\alpha})_2)\frac{\Psi(W,\phi,M(\tau,s)f_s)}
{\Psi(W,\phi,f_s)}
\\&=\omega(-I_{2l})
\gamma_{\psi_{\alpha}}(-1)\gamma(\psi_{\alpha})\epsilon(2s-1,1,\psi_{\alpha})\frac{1-q^{1-2s}}{1-q^{2s-2}}\frac{\Psi(W,\phi,M(\tau,s)f_s)}
{\Psi(W,\phi,f_s)}
\\&=\omega(-I_{2l})
\gamma_{\psi_{\alpha}}(-1)\gamma(\psi_{\alpha})q^{-s}\gamma_{\psi_{\alpha}}^{-1}(\varpi)\epsilon(2s-1,1,\psi_{\alpha})\frac{1-q^{1-2s}}{1-q^{2s-2}}
(q-1)\\&\quad\times[(1-q^{1-2s})^{-1}-1-(q-1)^{-1}]
\\&=\omega(-I_{2l})
\gamma(\psi_{\alpha})^{-1}\gamma_{\psi_{\alpha}}^{-1}(\varpi)q^{1/2-s},
\end{align*}
where we used $\epsilon(2s-1,1,\psi_{\alpha})=q^{2s-3/2}$ (\cite[Proposition~\S~23.5]{BH06}).
\end{proof}

We conclude:
\begin{itemize}
\item $\xi(x) = \tilde{\chi}^2(\iota(x))$, for all $x \in 1 + \mathfrak{p}_E$.
\item $\xi|_{\kappa_F^{\times}} = \tau_{\alpha}|_{\kappa_F^{\times}} \otimes (\mathrm{det}(\mathrm{Ind}_{W_E}^{W_F} 1_E))^{-1}$.
\item $\xi(\zeta) =  \omega(-I_{2l}) \gamma(\psi_{\alpha})^{-1}\gamma_{\psi_{\alpha}}^{-1}(\varpi) \gamma(s, \tau_{\alpha}, \psi_{\alpha})^{-1}  \lambda_{E/F}(\psi_{\alpha})^{-1}$.
\end{itemize}

This character $\xi$ depends on $\omega$ and $\alpha$ (and $\varpi$, though we have suppressed the notation $\varpi$ throughout), so we denote it $\xi_{\alpha}^{\omega}$. Thus $\varphi_1=\Ind_{W_E}^{W_F}(\xi_{\alpha}^{\omega})$, and
the Langlands paramter for $\pi = \pi_{\alpha}^{\omega}$ is $\varphi_{\pi} = \mathrm{Ind}_{W_E}^{W_F} \xi_{\alpha}^{\omega} \oplus \tau_{\alpha}$. This completes the proof of Theorem~\ref{thm:main param}.

\begin{rmk}\label{rmk:fullparameter}
Assume that $p \neq 2$ and $p \mid l$.  Then the computation of $\varpi_{\alpha, l}$ did not depend on whether or not $p$ divided $l$.  The following conditions then characterize $\Pi_1$ uniquely (following again the notation of \cite{AL14}): $\Pi_1$ is a simple supercuspidal representation of $\GL_{2l}$,  the uniformizer for $\Pi_1$ is $\varpi_{\alpha,l}$, the central character satisfies $\omega_{\Pi_1} = \tau_{\alpha}|_{F^{\times}}$, and $\gamma(s, \Pi_1, \psi_{\alpha}) = \omega(-I_{2l})
\gamma(\psi_{\alpha})^{-1}\gamma_{\psi_{\alpha}}^{-1}(\varpi) \gamma(s, \tau_{\alpha}, \psi_{\alpha})^{-1}$.  Then the work of \cite{BH14} gives the parameter of $\Pi_1$ explicitly, and together with $\tau_{\alpha}$ we get the full parameter of $\pi_{\alpha}^{\omega}$.
\end{rmk}

\subsection{The case $F=\mathbb{Q}_2$}\label{Q2}
In this case there is only one simple supercuspidal representation up to isomorphism, since $Z\subset I^+$ (see Section~\ref{simplesupercuspidaldefinition}) and $\kappa_F^{\times} / (\kappa_F^{\times})^2$ is trivial.  We denote this representation by $\pi$. Moreover, for convenience in the computations, we will fix a specific uniformizer $\varpi=2$. The end result does not depend on our choice of uniformizer, because of the aforementioned uniqueness. Also note that since in $\mathbb{Q}_2$, $\mathfrak{o}^{\times}=1+\mathfrak{p}$, a tamely ramified quasi-character is in fact unramified.
Throughout this section, we will regularly refer to \cite[Lemma 3.5]{Dani} for computations.

We still take the function $W$ of Section~\ref{section:The computation of the gamma-factor}, and we define
$f_s$ as in the $p\ne2$ case except that we now use the subgroup
\begin{align*}
\mathcal{N}=\left\{\left(\begin{array}{cc}1+\mathfrak{p}^3&\mathfrak{p}^2\\\mathfrak{p}^3&1+\mathfrak{p}^3\end{array}\right)\in\Sp_2\right\}.
\end{align*}
As opposed to the case of odd $p$, here we must prove directly that the cover is split over $\mathcal{N}$, in fact
the trivial section provides a splitting:
\begin{prop}
For any $v,v'\in \mathcal{N}$, $\langle v,1\rangle\langle v',1\rangle=\langle vv',1\rangle$.
\end{prop}
\begin{proof}
Write $v=\left(\begin{smallmatrix}         a & b \\         c & d       \end{smallmatrix}\right)$ and $v'=\left(\begin{smallmatrix}         a' & b' \\         c' & d'       \end{smallmatrix}\right)$. Since $d,a'\in F^{\times2}$, the result
$\langle v,1\rangle\langle v',1\rangle=\langle vv',1\rangle$ is clear from
\eqref{eq:Kubota formula} when $c=0$ or $c'=0$. Now assume $c,c'\ne0$.

Consider the case $ca'+dc'\ne0$.
Using $(x,-x)=1$, We can rewrite \eqref{eq:Kubota formula} in the form
\begin{align*}
\sigma(v,v')=(c,c')(-cc',ca'+dc')=(c,c')(-cc',c(a'+dc'c^{-1}))
=(-cc',a'+dc'c^{-1}).
\end{align*}
By symmetry, we may assume $c'=ct$ for $|t|\leq1$. Then
$(-cc',a'+dc'c^{-1})=(-t,a'+dt)$ and we can write
$a'+dt=1+t+r$ for $r\in\mathfrak{p}^3$. Now $1+r\in F^{\times2}$ hence
\begin{align*}
(-t,a'+dt)=(-t(1+r),a'+dt)=(1-(a'+dt),a'+dt)=1,
\end{align*}
where we used $(1-x,x)=1$.

Now assume $ca'+dc'=0$.  Then $c=-dc'{a'}^{-1}=-c't$ for $t\in F^{\times2}$ (because $d,a'\in F^{\times2}$), and
also $cb'+dd'\in F^{\times2}$. We obtain
\begin{align*}
\sigma(v,v')=(c,c')(-cc',cb'+dd')=(c,c')=(-c',c')=1.
\end{align*}
This completes the proof.
\end{proof}
Let $\phi$ be the characteristic function of $\mathfrak{o}$.
\begin{lem}\label{lemma:computation Psi with p equals 2}
$\Psi(W, \phi, f_s) =  2^{-1}\beta_{\psi}^{-2} \gamma_{\psi}(-1) \mathrm{vol}(\mathfrak{o})^2 \mathrm{vol}^{\times}(1 + \mathfrak{p}) \mathrm{vol}(\mathfrak{p}^3) \mathrm{vol}(\mathfrak{p})^{l-1}$.
\end{lem}
\begin{proof}
Writing the $dg$-integration over $\overline{B}_1$,
\begin{align*}
\Psi(W,\phi,f_s)=\int_{\overline{B}_1}\int_{R^{l,1}} \int_{F}
W({}^{\omega_{l-1,1}}(rxg)) \omega_{\psi}(g) \phi(x) f_s(g,1) \delta_{B_1}(b)\,d x \,d r \,d b.
\end{align*}
The support of $W$ is contained in $N_l Z I^+$, which equals $N_l I^+$ since $F = \mathbb{Q}_2$. Looking at
${}^{\omega_{l-1,1}}(rxg)$ with $b=\left(\begin{smallmatrix}a&0\\a^{-1}c&a^{-1}\end{smallmatrix}\right)\in \overline{B}_1$,
we see that $W$ vanishes unless $a\in 1+\mathfrak{p}$, $x,c\in\mathfrak{p}$ and the coordinates of
$r_1$ belong in $\mathfrak{p}$.  We may also assume $c\ne0$. Then $\delta_{B_1}(b)\equiv1$ and
\begin{align}\label{eq:W p equals 2}
W({}^{\omega_{l-1,1}}(rxg))=\psi^{-1}(a^{-1}c2^{-1}),
\end{align}
recalling that $\varpi = 2$.  By Lemma~\ref{lemma:computation weil} and since $\phi$ is the characteristic function of $\mathfrak{o}$,
\begin{align}\label{eq:phi_s p equals 2}
&\omega_{\psi}(b) \phi(x)=(a^{-1}, c) \beta_{\psi}^{-2}\gamma_{\psi}^{-1}(a)\gamma_{\psi}(-1)(2^{-1}\mathrm{vol}(\mathfrak{o})^2).
\end{align}
Here note that the measures from Lemma~\ref{lemma:computation weil} assign the volume $|2|^{1/2}\mathrm{vol}(\mathfrak{o})$ to $\mathfrak{o}$, because they are self dual with respect to $\psi_2$. Also for $a \in 1 + \mathfrak{p}$, $f_s(b,1)=0$ unless $c \in \mathfrak{p}^3$, in which case
\begin{align}\label{eq:f_s p equals 2}
f_s(\langle\left(\begin{smallmatrix}a&0\\0&a^{-1}\end{smallmatrix} \right) \left( \begin{smallmatrix}1&0\\c&1\end{smallmatrix}\right),1\rangle , 1) = (a^{-1},c)\gamma_{\psi}(a).
\end{align}
Therefore by \eqref{eq:W p equals 2}--\eqref{eq:f_s p equals 2},
\begin{align*}
\Psi(W, \phi, f_s) = 2^{-1}\beta_{\psi}^{-2} \gamma_{\psi}(-1) \mathrm{vol}(\mathfrak{o})^2 \mathrm{vol}^{\times}(1 + \mathfrak{p})
\mathrm{vol}(\mathfrak{p}^3) \mathrm{vol}(\mathfrak{p})^{l-1},
\end{align*}
as claimed.
\end{proof}
\begin{lem}\label{lemma:computation M(tau,s)f_s with p equals 2}
Let $a\in 1+\mathfrak{p}$, $0\ne c\in\mathfrak{p}$, and write $c\in c_02+ c_1 2^2 + \mathfrak{p}^3$ with $c_0, c_1\in\mathfrak{o}$. Then
\begin{align*}
[M(\tau,s)f_s]\left(\left(\begin{smallmatrix}a &  \\       a^{-1}c & a^{-1}\end{smallmatrix}\right),1 \right)=
\begin{cases}
2 \cdot A(\tau,\psi,s)  \gamma_{\psi}^{-1}(-ac) &c_0\neq0,\\
0 & c_0 = 0, c_1 \neq 0,
\end{cases}
\end{align*}
and if $c_0 = c_1 = 0$, then
\begin{align*}
&[M(\tau,s)f_s]\left(\left(\begin{smallmatrix}a &  \\       a^{-1}c & a^{-1}\end{smallmatrix}\right),1 \right)\\&=
2 A(\tau,\psi,s)    \gamma_{\psi}^{-1}(2)  (-c,a) \gamma_{\psi}^{-1}(a)   (1 + \psi(2^{-1}))  (L(2s-1,\tau^2)-1).
\end{align*}
Here
\begin{align*}
A(\tau,\psi,s) =     2^{-s} \tau(2) \mathrm{vol}^{\times}(1 + \mathfrak{p}^3).
\end{align*}
\end{lem}
\begin{proof}
By Lemma~\ref{lemma:computation M tau on b}, we need to compute
\begin{align}\label{integral}
 \mathrm{vol}(\mathfrak{o}^{\times}) \int_{F^{\times}}  (-uc,a)\tau(u^{-1}) |u|^{\frac{1}{2}-s}  \gamma_{\psi}^{-1}(u^{-1} a) f_s \left(
\begin{pmatrix}
1 & 0\\
a^2 u^{-1} +c & 1
\end{pmatrix}, -1\right) \,d^{\times} u.
\end{align}
We recall again that $F = \mathbb{Q}_2$, so that $q = 2$. In particular $\mathrm{vol}(\mathfrak{o}^{\times})=2^{-1/2}$.  We may assume $a^2 u^{-1} + c \in \mathfrak{p}^3$, since otherwise $f_s$ vanishes. Since $a \in 1 + \mathfrak{p}$, one can see that the condition $a^2 u^{-1} + c \in \mathfrak{p}^3$ then implies that $u^{-1} \equiv -c \ (\mathrm{mod} \ \mathfrak{p}^3)$.

First assume that $c \in \mathfrak{p} \setminus \mathfrak{p}^2$.  Then $u^{-1} = -c \cdot v$, with $v \in 1 + \mathfrak{p}^2$.  One can see that \eqref{integral} simplifies to
\begin{align*}
\tau(-1)  2^{-s} \int_{1 + \mathfrak{p}^2}  \tau(-c v) \gamma_{\psi}^{-1}(-c v a) (v^{-1}, a)\, d^{\times} v.
\end{align*}
Since $\tau$ is tamely ramified, we have $\tau(-1)=1$ whence $\tau(-cv) = \tau(2)$, and we get
\begin{align*}
2^{-s} \tau(2) \int_{1 + \mathfrak{p}^2} \gamma_{\psi}^{-1}(-c v a) (v, a) \,d^{\times} v.
\end{align*}
Also because $\gamma_{\psi}^{-1}(-cva) = \gamma_{\psi}^{-1}(-ca) \gamma_{\psi}^{-1}(v) (-ca, v)$, this last integral equals
\begin{align*}
2^{-s} \tau(2)  \gamma_{\psi}^{-1}(-ca) \int_{1 + \mathfrak{p}^2} \gamma_{\psi}^{-1}(v) (-c, v) \,d^{\times} v.
\end{align*}
The $d^{\times}v$-integral is invariant under $1 + \mathfrak{p}^3$, and since $(1 + \mathfrak{p}^2) / (1 + \mathfrak{p}^3) = \{ 1, 5 \}$, using \cite[Lemma 3.5]{Dani} we see that it equals $2 \mathrm{vol}^{\times}(1 + \mathfrak{p}^3)$. In summary, \eqref{integral} equals
\begin{align*}
2^{-s} \tau(2)  \gamma_{\psi}^{-1}(-ca) \cdot 2 \mathrm{vol}^{\times}(1 + \mathfrak{p}^3).
\end{align*}

We now assume that $c \in \mathfrak{p}^2 \setminus \mathfrak{p}^3$.  Then $u^{-1} = -c \cdot v$, with $v \in 1 + \mathfrak{p}$.  The same analysis as in $\mathfrak{p} \setminus \mathfrak{p}^2$ gives us that \eqref{integral} simplifies to
\begin{align*}
 2^{1/2-2s} \tau^2(2)  \gamma_{\psi}^{-1}(-ca) \int_{1 + \mathfrak{p}} \gamma_{\psi}^{-1}(v) (-c, v) \,d^{\times} v.
\end{align*}
Again the $d^{\times} v$-integral is invariant under $1 + \mathfrak{p}^3$, but now $(1 + \mathfrak{p}) / (1 + \mathfrak{p}^3) = \{ \pm 1, \pm 5 \}$, and by \cite[Lemma 3.5]{Dani} this integral vanishes.

Finally assume $c \in \mathfrak{p}^3$. Then we must have $u^{-1} \in \mathfrak{p}^3$ and \eqref{integral} equals
\begin{align*}
2^{-1/2} \tau(-1)  (-c, a) \displaystyle\sum_{n=3}^{\infty} \int_{\mathfrak{o}^{\times}} \tau(2^{n} v) |2^{-n} v^{-1}|^{\frac{1}{2}-s} \gamma_{\psi}^{-1}(2^n v a)  (2^{-n} v^{-1},a) \,d^{\times} v.
\end{align*}
Since $\tau$ is tamely ramified and $F = \mathbb{Q}_2$ , we obtain
\begin{align*}
&2^{-1/2}  (-c,a) \displaystyle\sum_{n=3}^{\infty} \tau^{n}(2) 2^{n(1/2 - s)}  \int_{1 + \mathfrak{p}} \gamma_{\psi}^{-1}(2^n v a)  (2^{-n} v^{-1}, a) \,d^{\times} v\\&
=  2^{-1/2} (-c,a) \gamma_{\psi}^{-1}(a) \displaystyle\sum_{n=3}^{\infty} \tau^{n}(2) 2^{n(1/2 - s)}   \gamma_{\psi}^{-1}(2^n)  \int_{1 + \mathfrak{p}}  \gamma_{\psi}^{-1}(v) (2^n, v) \,d^{\times} v.
\end{align*}
As above, we compute the $d^{\times}v$-integral using the representatives $\{\pm 1, \pm 5 \}$. If $n$ is even, by \cite[Lemma 3.5]{Dani},
\begin{align*}
\int_{1 + \mathfrak{p}}  \gamma_{\psi}^{-1}(v) (2^n, v) \,d^{\times} v = \mathrm{vol}^{\times}(1 + \mathfrak{p}^3) (\gamma_{\psi}^{-1}(1) + \gamma_{\psi}^{-1}(-1) + \gamma_{\psi}^{-1}(5) + \gamma_{\psi}^{-1}(-5)) = 0.
\end{align*}
For odd $n$, by \textit{loc. cit.} the integral equals $\mathrm{vol}^{\times}(1 + \mathfrak{p}^3) (2 + 2 \psi(2^{-1}))$.
Also $\gamma_{\psi}(2)^{2m+1} = \gamma_{\psi}(2)$, for any integer $m$ (since $(2,2)=1$).
Therefore \eqref{integral} equals
\begin{align*}
 2^{-1/2}   \gamma_{\psi}^{-1}(2) (-c,a) \gamma_{\psi}^{-1}(a) \mathrm{vol}^{\times}(1 + \mathfrak{p}^3) (2 + 2 \psi(2^{-1})) \tau(2) 2^{1/2-s} \left( \frac{1}{1 - \tau^2(2) 2^{1-2s}} - 1 \right).
\end{align*}
Because $\tau$ is unramified, the factor in parentheses is $L(2s-1,\tau^2)-1$.
This completes the proof.
\end{proof}
\begin{lem}\label{lemma:computation Psi M(tau,s)f_s with p equals 2}
\begin{align*}
&\Psi(W,\phi,M(\tau,s)f_s)=2\Psi(W, \phi, f_s)A(\tau,\psi,s)(1 + \psi(2^{-1}))\gamma_{\psi}^{-1}(2)\frac{-1 + \tau^2(2) 2^{2-2s}}{1 - \tau^2(2) 2^{1-2s}}.
\end{align*}
\end{lem}
\begin{proof}
As in the proof of Lemma~\ref{lemma:computation Psi with p equals 2}, we compute the integral using
$\overline{B}_1$, then for $b$ given by \eqref{b}, the support of $W$ implies
$a \in 1+\frak{p}$ and $c \in \frak{p}$, the choice of $\phi$ implies $x \in \mathfrak{p}$, we can assume $c\ne0$, and we deduce
\eqref{eq:W p equals 2} and \eqref{eq:phi_s p equals 2}.
Since our uniformizer $\varpi$ is $2$, and by virtue of Lemma~\ref{lemma:computation M(tau,s)f_s with p equals 2},
the integral $\Psi(W,M(\tau,s)f_s)$ becomes a constant
\begin{align}\label{D constant}
2A(\tau,\psi,s)\mathrm{vol}(\mathfrak{p})^{l-1}\gamma_{\psi}(-1)(2^{-1}\mathrm{vol}(\mathfrak{o})^2)\beta_{\psi}^{-2}
\end{align}
multiplied by the sum of integrals
\begin{align}\nonumber
&\int_{1+\mathfrak{p}}\int_{\mathfrak{p} \setminus \mathfrak{p}^2}
\psi^{-1}(a^{-1}c2^{-1})
(a^{-1}, c)  \gamma_{\psi}^{-1}(a)
\gamma_{\psi}^{-1}(-ac)\,dc\,da\\\label{int1 p = 2}
&=\int_{1+\mathfrak{p}}\int_{\mathfrak{p} \setminus \mathfrak{p}^2}
\psi^{-1}(a^{-1}c2^{-1})
\gamma_{\psi}^{-1}(-c)\,dc\,da
\end{align}
and
\begin{align}\label{int2 p = 2}
&\gamma_{\psi}^{-1}(2)(1 + \psi(2^{-1}))(L(2s-1,\tau^2)-1)  \int_{1+\mathfrak{p}}\int_{\mathfrak{p}^3}\psi^{-1}(a^{-1}c2^{-1})\,dc\,da.
\end{align}
Note that we used $\gamma_{\psi}^{-1}(-ac)=\gamma_{\psi}^{-1}(-c)\gamma_{\psi}(a)(c,a)$ to obtain
\eqref{int1 p = 2}, and $(-1,a)\gamma_{\psi}^{-1}(a)=\gamma_{\psi}(a)$ for \eqref{int2 p = 2}.

Regarding \eqref{int1 p = 2}, since $\psi$ is of level $1$ and $|a|=1$, $\psi^{-1}(a^{-1}c2^{-1})=-1$ for all $0\ne c \in \frak{p}$.
Moreover, since the additive volume of $\mathfrak{p} \setminus \mathfrak{p}^2$ equals $(2^{-1/2} - 2^{-3/2})$ and the multiplicative volume of
$\mathfrak{o}^{\times}$ is $1$,
\begin{align*}
\int_{\mathfrak{p} \setminus \mathfrak{p}^2} \gamma_{\psi}^{-1}(-c)\, dc
&=(2^{-1/2} - 2^{-3/2})\gamma_{\psi}^{-1}(2) \mathrm{vol}^{\times}(1 + \mathfrak{p}^3) (\sum_{v \in \{ \pm 1, \pm 5 \} } \gamma_{\psi}^{-1}(-v) (2, -v))\\
&=(2^{-1/2} - 2^{-3/2})\gamma_{\psi}^{-1}(2) \mathrm{vol}^{\times}(1 + \mathfrak{p}^3) 2(1 + \psi(2^{-1})),
\end{align*}
where the second equality follows using \cite[Lemma 3.5]{Dani}. It follows that \eqref{int1 p = 2} equals
\begin{align}\label{int11 p = 2}
-2\mathrm{vol}^{\times}(1+\mathfrak{p})(2^{-1/2} - 2^{-3/2})\gamma_{\psi}^{-1}(2) \mathrm{vol}^{\times}(1 + \mathfrak{p}^3)(1 + \psi(2^{-1})).
\end{align}
Also note that in \eqref{int2 p = 2}, $\psi^{-1}(a^{-1}c2^{-1})$ is identically $1$, so that the integral in \eqref{int2 p = 2} is a constant.

Combining \eqref{D constant}, \eqref{int2 p = 2} and \eqref{int11 p = 2}, we see that $\Psi(W,\phi,M(\tau,s)f_s)$ equals
\begin{align*}
&A(\tau,\psi,s)\mathrm{vol}(\mathfrak{p})^{l-1}\gamma_{\psi}(-1)\mathrm{vol}(\mathfrak{o})^2\beta_{\psi}^{-2}
\mathrm{vol}^{\times}(1+\mathfrak{p})(1 + \psi(2^{-1}))\gamma_{\psi}^{-1}(2)\\&\times[-2(2^{-1/2} - 2^{-3/2})\mathrm{vol}^{\times}(1 + \mathfrak{p}^3)+(L(2s-1,\tau^2)-1)\mathrm{vol}(\mathfrak{p}^3)].
\end{align*}
To continue, since $q=2$,
\begin{align*}
\mathrm{vol}^{\times}(1 + \mathfrak{p}^3) = 2^{-2},\quad \mathrm{vol}(\mathfrak{p}^3) = 2^{-5/2}, \quad
L(2s-1,\tau^2)=(1 - \tau^2(2) 2^{1-2s})^{-1}.
\end{align*}
Thus
\begin{align*}
&-2(2^{-1/2} - 2^{-3/2})\mathrm{vol}^{\times}(1 + \mathfrak{p}^3)+(L(2s-1,\tau^2)-1)\mathrm{vol}(\mathfrak{p}^3)
=\mathrm{vol}(\mathfrak{p}^3)\frac{-1 + \tau^2(2) 2^{2-2s}}{1 - \tau^2(2) 2^{1-2s}}.
\end{align*}
Comparing this to the result of Lemma~\ref{lemma:computation Psi with p equals 2} yields the result.
\end{proof}
\begin{thm}\label{thm:p=2}
Assume that $\tau$ is tamely ramified. Then $\gamma(s,\pi\times\tau,\psi) = \tau(2) 2^{1/2-s}$.
\end{thm}
\begin{proof}
Since by Lemma~\ref{lemma:computation Psi with p equals 2}, $\Psi(W, \phi, f_s)$ is nonzero, by definition
\begin{align*}
\gamma(s,\pi\times\tau,\psi)&= \pi(-I_{2l})\tau(-1)^l \gamma(s,\tau,\psi)c(s,l,\tau)\frac{C(s,\tau,\psi)\Psi(W, \phi, M(\tau,s) f_s)}{\Psi(W, \phi, f_s)}.
\end{align*}
Recall that by \eqref{C factor},
\begin{align*}
C(s,\tau,\psi)=\gamma_{\psi}(-1)\gamma(\psi)\frac{\gamma(2s-1,\tau^2,\psi_2)}{\gamma(s,\tau,\psi)}.
\end{align*}
We also have $\pi(-I_{2l})=\tau(-1)=1$, because $-1\in1+\mathfrak{p}$ ($F=\mathbb{Q}_2$). Plugging in Lemma~\ref{lemma:computation Psi M(tau,s)f_s with p equals 2} and because $A(\tau,\psi,s)=2^{-2-s}\tau(2)$ (note that $\mathrm{vol}^{\times}(1 + \mathfrak{p}^3)=2^{-2}$) we obtain
\begin{align*}
\gamma(s,\pi\times\tau,\psi)&= c(s,l,\tau)\gamma_{\psi}(-1)\gamma(\psi)\gamma(2s-1,\tau^2,\psi_2)
\\&\quad\times2^{-1-s}\tau(2)(1 + \psi(2^{-1}))\gamma_{\psi}^{-1}(2)\frac{-1 + \tau^2(2) 2^{2-2s}}{1 - \tau^2(2) 2^{1-2s}}.
\end{align*}

According to the twisting property of the standard $\gamma$-factor (see e.g., \cite[\S~23.5, Lemma 1]{BH06}),
\begin{align*}
\gamma(2s-1,\tau^2,\psi_2)=|2|^{2s-3/2}\tau^2(2)\gamma(2s-1,\tau^2,\psi)=c(s,l,\tau)^{-1}2^{1/2}\gamma(2s-1,\tau^2,\psi).
\end{align*}
Also by definition,
\begin{align*}
\gamma(2s-1, \tau^2, \psi) = \epsilon(2s-1, \tau^2, \psi) \frac{L(\tau^{-2}, 2-2s)}{L(\tau^2, 2s-1)}
=\epsilon(2s-1, \tau^2, \psi) \frac{1 - \tau^2(2) 2^{1-2s}}{1 - \tau^{-2}(2) 2^{2s-2}}.
\end{align*}
Then using the value of $\epsilon(2s-1, \tau^2,\psi)$ given by \cite[\S~23.5, Proposition]{BH06},
\begin{align*}
\epsilon(2s-1, \tau^2, \psi) \frac{1 - \tau^2(2) 2^{1-2s}}{1 - \tau^{-2}(2) 2^{2s-2}}
\times \frac{-1 + \tau^2(2) 2^{2-2s}}{1 - \tau^2(2) 2^{1-2s}}=2^{1/2}.
\end{align*}
Therefore
\begin{align*}
\gamma(s,\pi\times\tau,\psi)&= \gamma_{\psi}(-1)\gamma(\psi)
(1 + \psi(2^{-1}))\gamma_{\psi}^{-1}(2)\tau(2)2^{-s}.
\end{align*}

To complete the computation note that $\gamma_{\psi}(2)=1$ (see \cite[(3.23)]{Dani}),
$\gamma_{\psi}(-1)=\psi(-2^{-1})$ (\cite[(3.21)]{Dani}), and
\begin{align*}
\psi(-2^{-1})\gamma(\psi)(1 + \psi(2^{-1}))=2^{1/2}.
\end{align*}
This can be verified for both choices of $\psi$, which are $\psi(x) = e^{\pm\pi i x}$, using the formulas in
\cite[p.~370]{Rao} for $\gamma(\psi)$.
\end{proof}

\bibliographystyle{alpha-abbrvsort}
\bibliography{bib}

\end{document}